\documentclass[12pt]{article}
\usepackage{graphicx}
\usepackage{amsmath}
\usepackage{amsfonts}
\usepackage{amssymb}
\newtheorem{theorem}{Theorem}

\newtheorem{corollary}[theorem]{Corollary}

\newtheorem{definition}[theorem]{Definition}

\newtheorem{lemma}[theorem]{Lemma}

\newtheorem{proposition}[theorem]{Proposition}
\newtheorem{remark}[theorem]{Remark}

\newenvironment{proof}[1][Proof]{\textbf{#1.} }{\ \rule{0.5em}{0.5em}}

\begin{document}

\title{On a new method for controlling exponential processes}
\author{O. Kounchev and H. Render}
\maketitle
\begin{abstract}
Unlike the classical polynomial case there has not been invented up to very
recently a tool similar to the Bernstein-B\'{e}zier representation which would
allow us to control the behavior of the exponential polynomials. The
exponential analog to the classical Bernstein polynomials has been introduced
in the recent authors' paper \cite{AKR}, and this analog retains all basic
propeties of the classical Bernstein polynomials. The main purpose of the
present paper is to contribute in this direction, by proving some important
properties of the \emph{Bernstein exponential operator} which has been
introduced in \cite{AKR}. We also fix our attention upon some special type of
exponential polynomials which are particularly important for the further
development of theory of representation of Multivariate data.

The first-named author has been supported by the Greek-Bulgarian bilateral
project B-Gr17, 2005-2008. The second-named author is partially supported by
Grant BFM2003-06335-C03-03 of the D.G.I. of Spain. Both authors acknowledge
support within the project ``Institutes Partnership'' with the Alexander von
Humboldt Foundation, Bonn.

O. Kounchev: Institute of Mathematics and Informatics, Bulgarian Academy of
Sciences, 8 Acad. G. Bonchev Str., 1113 Sofia, Bulgaria. {kounchev@gmx.de} H.
Render: Departamento de Matem\'{a}ticas y Computaci\'{o}n, Universidad de La
Rioja, Edificio Vives, Luis de Ulloa s/n., 26004 Logro\~{n}o, Espa\~{n}a. {render@gmx.de}
\end{abstract}

\section{Introduction}

With the great increase in the number of businesses making a presence on the
Internet and the increase in the number of cyber-customers, the chances of
computer security attacks increase daily. How much should an enterprise or
organization budget to defend against Internet attacks?

As it is well know the processes modelled by ordinary differential equations
and the stochastic differential equations describe many important practical
events, and thus provide very efficient models for processes coming from
physics, finance, etc. As is well known their solutions may be approximated
efficiently by means of finite linear combinations of exponential functions,
called sometimes exponential polynomials. Let us mention some recent
applications to exponential smoothing used for time series models which fit
best for forecasting \emph{Internet Security attacks}, \cite{korzyk}.

Going further in this direction, if one considers a process which experiences
jumps or unsmoothness, then it may be approximated by means of piecewise
exponential polynomials.

For the control and design of polynomials the famous Bernstein polynomials
provide a very efficient way to control their behaviour. Their further
development as the Bezier curves provides a very efficient method to design a
prescribed form, which is an indispensable tool for Computer Design. In a
similar way the theory of $B-$splines provides us with an indispensable tool
for the control and design of the classical splines.

Let us add to the above that for the purposes of fast \emph{recognition},
\emph{representation} and \emph{compression} of curves and surfaces such tools
as the B\'{e}zier curves and $B-$splines curves are very important. Since many
of the real processes produce \emph{observational} and \emph{surveillance
data} which carry exponential character, it is clear that it is very important
to have an efficient tools for their \emph{representation}, \emph{design} and
\emph{control}.

It is curious to mention that unlike the classical polynomial case there has
not been invented up to very recently a tool similar to the Bernstein-B\'{e}%
zier representation which would allow us to control the behavior of the
exponential polynomials. The exponential analog to the classical Bernstein
polynomials has been introduced in the recent authors' paper \cite{AKR}, and
this analog retains all basic propeties of the classical Bernstein polynomials.

The main purpose of the present paper is to contribute in this direction, by
proving some important properties of the \emph{Bernstein exponential operator}
which has been introduced in \cite{AKR}. We also fix our attention upon some
special type of exponential polynomials which are particularly important for
the further development of theory of representation of Multivariate data.

Let us recall at first shortly the concept of B\'{e}zier curves and its
relationship to the Bernstein polynomials: Let $b_{0},...,b_{n}$ be vectors
either in $\mathbb{R}^{2}$ or $\mathbb{R}^{3}$ and $t\in\mathbb{R},$ and
interpret $b_{k}$ as a constant curve, i.e. that $b_{k}^{0}\left(  t\right)
:=b_{k}$ for $k=0,...,n.$ Then define new curves $b_{k}^{1}\left(  t\right)  $
for $k=0,...,n-1$ by $b_{k}^{1}\left(  t\right)  =\left(  1-t\right)
b_{k}^{0}\left(  t\right)  +tb_{k+1}^{0}\left(  t\right)  .$ Repeating this
process one arrives at new curves
\[
b_{k}^{r}\left(  t\right)  =\left(  1-t\right)  b_{k}^{r-1}\left(  t\right)
+tb_{k+1}^{r-1}\left(  t\right)  \text{ for }k=0,...,n-r
\]
In the last step, i.e. for $r=n,$ one finally obtains exactly one curve
\[
b_{0}^{n}\left(  t\right)  =\left(  1-t\right)  b_{0}^{n-1}\left(  t\right)
+tb_{1}^{n-1}\left(  t\right)  ,
\]
the so-called \emph{B\'{e}zier curve}. The polygon formed by $b_{0},...,b_{n}$
is called the \emph{B\'{e}zier polygon or control polygon}. The B\'{e}zier
curve has the property that the curve $b_{0}^{n}\left(  t\right)  $ is in the
convex hull generated by the points $b_{0},...,b_{n}.$ Moreover the two points
$b_{0}$ and $b_{n}$ are fixed, i.e. $b_{0}^{n}\left(  0\right)  =b_{0}$ and
$b_{0}^{n}\left(  1\right)  =b_{n}.$ An explicit form for the B\'{e}zier curve
is
\[
b_{0}^{n}\left(  t\right)  =\sum_{k=0}^{n}b_{k}p_{n,k}\left(  t\right)
\]
where $p_{n,k}\left(  t\right)  :=\binom{n}{k}t^{k}\left(  1-t\right)  ^{n-k}$
are called the \emph{Bernstein basis polynomials.} In the sequel we shall
focus on generalizations of Bernstein basis polynomials which have arisen
recently in Computer Aided Geometric Design for modeling parametric curves.
Instead of the basic polynomials $1,x,....,x^{n}$ one consider different
systems of basic functions $f_{0},...,f_{n}$, e.g.
\[
1,x,...,x^{n-2},\cos x,\sin x,
\]
which are better adapted to curves in spherical coordinates, see e.g.
\cite{Cost00}, \cite{MPS97}, \cite{Zhan96} and \cite{CLM}. In mathematical
terms it will be required that the linear span of the basis functions
$f_{0},...,f_{n}$ forms a an extended Chebyshev system. Recall that a subspace
$U_{n}$ of $C^{n}\left(  I\right)  $ (the space of $n$-times continuously
differentiable complex-valued functions on a interval $I$) is called an
\emph{extended Chebyshev system for a subset }$A\subset I$, if $U_{n}$ has
dimension $n+1$ and each non-zero function $f\in U_{n}$ vanishes at most $n$
times on the subset $A$ (counted with multiplicities). A system $p_{n,k}\in
U_{n},k=0,...,n$, is a \emph{Bernstein-like basis} for $U_{n}$ relative to
$a,b\in I$, if for each $k=0,...,n$ the function $p_{n,k}$ has a zero of order
$k$ at $a$, and a zero of order $n-k$ at $b$.

In the following we shall consider Bernstein basis polynomials for the space
of \emph{exponential polynomials} $E_{\left(  \lambda_{0},...,\lambda
_{n}\right)  }$ (induced by a linear differential operator $L)$ defined by
\begin{equation}
E_{\left(  \lambda_{0},...,\lambda_{n}\right)  }:=\left\{  f\in C^{\infty
}\left(  \mathbb{R}\right)  :Lf=0\right\}  , \label{Espace}%
\end{equation}
where $\lambda_{0},...,\lambda_{n}$ are complex numbers, and $L$ is the linear
differential operator with constant coefficients defined by
\begin{equation}
L:=L\left(  \Lambda\right)  :=\left(  \frac{d}{dx}-\lambda_{0}\right)
....\left(  \frac{d}{dx}-\lambda_{n}\right)  . \label{eqdefL}%
\end{equation}
Exponential polynomials are sometimes called $L$\emph{-polynomials, }and they
provide natural generalization of classical, trigonometric, and hyperbolic
polynomials (see \cite{Schu83}), and the so-called $\mathcal{D}$-polynomials
considered in \cite{MoNe00}. Exponential polynomials arise naturally in the
context of a class of multivariate splines, the so-called polysplines, see
\cite{oggybook}, \cite{KounchevRenderJAT}). Important in this context are
exponential splines associated with linear differential operators $L_{2s+2}$
of order $2s+2$ of the form
\[
L_{2s+2}=\left(  \frac{d}{dx}-\lambda\right)  ^{s+1}\left(  \frac{d}{dx}%
+\mu\right)  ^{s+1}%
\]
which are parametrized by real numbers $\lambda$ and $\mu$. According to the
above notation $\Lambda_{n}=\left(  \lambda_{0},...,\lambda_{n}\right)  $ we
set $n=2s+1$ and define the vector
\begin{equation}
\Lambda_{2s+1}\left(  \lambda,\mu\right)  :=\left(  \underset{s+1}%
{\underbrace{\lambda,\lambda,...,\lambda,}}\underset{s+1}{\underbrace{\mu
,\mu,...,\mu}}\right)  \label{eigenval}%
\end{equation}
containing $s+1$ times $\lambda$ and $s+1$ times $\mu.$

Let us return to the general theory of exponential polynomials, and let us
recall the general fact (cf. \cite{Micc76}) that for $\Lambda_{n}=\left(
\lambda_{0},...,\lambda_{n}\right)  \in\mathbb{C}^{n+1}$ there exists a unique
function $\Phi_{\Lambda_{n}}\in E_{\left(  \lambda_{0},...,\lambda_{n}\right)
}$ such that $\Phi_{\Lambda_{n}}\left(  0\right)  =....=\Phi_{\Lambda_{n}%
}^{\left(  n-1\right)  }\left(  0\right)  =0$ and $\Phi_{\Lambda_{n}}^{\left(
n\right)  }\left(  0\right)  =1.$ We shall call $\Phi_{\Lambda_{n}}$ the
\emph{fundamental function} in $E_{\left(  \lambda_{0},...,\lambda_{n}\right)
}.$ An explicit formula for $\Phi_{\Lambda_{n}}$ is
\begin{equation}
\Phi_{\Lambda_{n}}\left(  x\right)  :=\left[  \lambda_{0},...,\lambda
_{n}\right]  e^{xz}=\frac{1}{2\pi i}\int_{\Gamma_{r}}\frac{e^{xz}}{\left(
z-\lambda_{0}\right)  ...\left(  z-\lambda_{n}\right)  }dz \label{defPhi}%
\end{equation}
where $\left[  \lambda_{0},...,\lambda_{n}\right]  $ denotes the divided
difference, and $\Gamma_{r}$ is the path in the complex plane defined by
$\Gamma_{r}\left(  t\right)  =re^{it}$, $t\in\left[  0,2\pi\right]  $,
surrounding all the scalars $\lambda_{0},...,\lambda_{n}$. The fundamental
function $\Phi_{\Lambda_{n}}$ is an important tool in the spline theory based
on exponential polynomials (see \cite{Micc76}) and in the wavelet analysis of
exponential polynomials, see \cite{dDR93}, \cite{LySc93}, \cite{oggybook} ,
\cite{kounchevrenderappr}, \cite{kounchevrenderpams}.

In Section \ref{S1} we provide some basic results about the fundamental
function, in particular we derive the Taylor expansion of the fundamental
function $\Phi_{\left(  \lambda_{0},...,\lambda_{n}\right)  }$. Section
\ref{S2} is devoted to the analysis of the fundamental function with respect
to the eigenvalues (\ref{eigenval}), and in order to have a short notation we
set using (\ref{defPhi})
\[
\Phi_{2s+1}\left(  x\right)  :=\Phi_{\Lambda_{2s+1}\left(  \lambda,\mu\right)
}\text{ for }\lambda=1\text{ and }\mu=-1.
\]
We shall give explicit formula for $\Phi_{2s+1}\left(  x\right)  .$ So far it
turns out that the following recursion formula is more important:
\[
\Phi_{2s+3}\left(  x\right)  =x^{2}\frac{1}{4s\left(  s+1\right)  }\Phi
_{2s-1}\left(  x\right)  -\frac{2s+1}{2s+2}\Phi_{2s+1}\left(  x\right)  .
\]
We shall prove the non-trivial fact that for each fixed $x>0$ the sequence
\begin{equation}
\frac{\Phi_{2s}\left(  x\right)  }{\Phi_{2s-}\left(  x\right)  }\rightarrow1
\label{eqphic}%
\end{equation}
for $s\rightarrow\infty;$ here $\Phi_{2s}\left(  x\right)  $ is the
fundamental function with respect to the vector $\Lambda_{2s}$ consisting of
$s+1$ eigenvalues $1$ and $s$ eigenvalues $-1,$ and similarly $\Phi
_{2s-}\left(  x\right)  $ is the fundamental function with respect the vector
$\Lambda_{2s}$ consisting of $s$ eigenvalues $1$ and $s+1$ eigenvalues $-1.$

In Section \ref{S3} we shall determine some generating functions: we prove
that
\[
\sum_{s=0}^{\infty}\Phi_{2s+1}\left(  x\right)  \cdot y^{s}=\frac{1}%
{\sqrt{y+1}}\sinh\left(  x\sqrt{y+1}\right)
\]
and that
\[
\sum_{s=0}^{\infty}\Phi_{s}\left(  x\right)  \cdot y^{s}=\frac{1+y}%
{\sqrt{y^{2}+1}}\sinh\left(  x\sqrt{y^{2}+1}\right)  +\cosh\left(
x\sqrt{y^{2}+1}\right)  .
\]

In Section \ref{S4} we give a detailed introduction to the notion of a
Bernstein basis $p_{n,k}\left(  x\right)  ,k=0,...,n,$ in the setting of
exponential polynomials with arbitrary (complex) eigenvalues. We shall give a
new proof of the result that there exists a Bernstein basis $p_{n,k}\left(
x\right)  ,k=0,...,n,$ in $E_{\left(  \lambda_{0},...,\lambda_{n}\right)  }$
for points $a\neq b$ if and only if $E_{\left(  \lambda_{0},...,\lambda
_{n}\right)  }$ is a Chebyshev space with respect to $a,b.$ In the case that
$\lambda_{0},...,\lambda_{n}$ are real it is well known that $E_{\left(
\lambda_{0},...,\lambda_{n}\right)  }$ is a Chebyshev space with respect to
the interval $\left[  a,b\right]  $ and in this case $p_{n,k}\left(  x\right)
$ may be chosen strictly positive for $x$ in the open interval $\left(
a,b\right)  $.

In Section \ref{S5} we shall derive recursion formulas for the Bernstein basis
$p_{n,k}\left(  x\right)  ,k=0,...,n$ for the special system of eigenvalues
$\Lambda_{2s+1}$ which have been used in Section \ref{S2}.

A remarkable result was proved recently in \cite{CMP04}, \cite{Mazu05} for
certain classes of extended Chebyshev systems $U_{n}$: Assume that the
constant function $1$ is in $U_{n};$ clearly then there exist coefficients
$\alpha_{k},k=0,...,n,$ such that $1=\sum_{k=0}^{n}\alpha_{k}p_{n,k},$ since
$p_{n,k},k=0,...,n$, is a basis. The \emph{normalization property} proved in
\cite{CMP04} and \cite{Mazu05} says that the coefficients $\alpha_{k}$ are
\emph{positive}. It seems that the paper \cite{AKR} addresses for the first
time the question whether one can construct a Bernstein-type operator based on
a Bernstein basis $p_{n,k},k=0,...,n$, in the context of exponential
polynomials, i.e. operators of the form
\begin{equation}
B_{n}\left(  x\right)  :=\left[  B_{n}f\right]  \left(  x\right)  :=\sum
_{k=0}^{n}\alpha_{k}f\left(  t_{k}\right)  p_{n,k}\left(  x\right)
\label{defBN}%
\end{equation}
where the coefficients $\alpha_{0},...,\alpha_{n}$ should be positive and the
knots $t_{0},...,t_{n}$ in the interval $\left[  a,b\right]  .$ In \cite{AKR}
the following basic result was proven:

\begin{theorem}
\label{Thm1}Assume that $\lambda_{0},...,\lambda_{n}$ are \emph{real} and
$\lambda_{0}\neq\lambda_{1}.$ Then there exist unique points $t_{0}%
<t_{1}<...<t_{n}$ \emph{in the interval} $\left[  a,b\right]  $ and unique
\emph{positive} coefficients $\alpha_{0},...,\alpha_{n}$ such that the
operator $B_{n}:C\left[  a,b\right]  \rightarrow E_{\left(  \lambda
_{0},...,\lambda_{n}\right)  }$ defined by (\ref{defBN}) has the following
reproduction property
\begin{equation}
\left[  B_{n}\left(  e^{\lambda_{0}\cdot}\right)  \right]  \left(  x\right)
=e^{\lambda_{0}x}\text{ and }\left[  B_{n}\left(  e^{\lambda_{1}\cdot}\right)
\right]  \left(  x\right)  =e^{\lambda_{1}x}. \label{eqnorming}%
\end{equation}
\end{theorem}

The positivity of the coefficients $\alpha_{0},...,\alpha_{n}$ is related to
the above-mentioned normalization property. Theorem \ref{Thm1} says that
property (\ref{eqnorming}) can be used for defining knots $t_{0},...,t_{n}$
and weights $\alpha_{0},...,\alpha_{n}$ for an operator of the form
(\ref{defBN}). In the classical polynomial case this means that the Bernstein
operator $B_{n}$ on $[0,1]$ has the property that $B_{n}\left(  1\right)  =1$
and $B_{n}\left(  x\right)  =x$ for the constant function $1$ and the identity function.

It follows from the above construction that the operator $B_{n}$ defined by
(\ref{defBN}) and satisfying (\ref{eqnorming}) is a \emph{positive} operator.
Using a Korovkin-type theorem for extended Chebyshev systems the following
sufficient criterion for the uniform convergence of $B_{n}f$ to $f$, for each
$f\in C\left[  a,b\right]  $ has been given in \cite{AKR}. Here we use the
more precise but lengthy notation $p_{\left(  \lambda_{0},...,\lambda
_{n}\right)  ,k}$ instead of $p_{n,k}$ for $k=0,...,n.$

\begin{theorem}
\label{ThmCon}Let $\lambda_{0},\lambda_{1},\lambda_{2}$ be pairwise distinct
real numbers and let $\Lambda_{n}=\left(  \lambda_{0},\lambda_{1}%
,....,\lambda_{n}\right)  \in\mathbb{R}^{n+1}$ with possibly variable
$\lambda_{j}=\lambda_{j}\left(  n\right)  $ for $j=3,...,n$. For each natural
number $n\geq2,$ and each $k\leq n$, define the numbers $a\left(  n,k\right)
$ and $b\left(  n,k\right)  $ as follows:
\begin{align}
a\left(  n,k\right)   &  :=\lim_{x\rightarrow b}\frac{p_{\left(  \lambda
_{0},\lambda_{2},....,\lambda_{n}\right)  ,k}\left(  x\right)  }{p_{\left(
\lambda_{1},\lambda_{2},....,\lambda_{n}\right)  ,k}\left(  x\right)  },\text{
}\label{assump1}\\
b\left(  n,k\right)   &  :=\lim_{x\rightarrow b}\frac{p_{\left(  \lambda
_{0},\lambda_{1},\lambda_{3},....,\lambda_{n}\right)  ,k}\left(  x\right)
}{p_{\left(  \lambda_{1},\lambda_{2},....,\lambda_{n}\right)  ,k}\left(
x\right)  }. \label{assump2}%
\end{align}
Assume that for $n\longrightarrow\infty$ uniformly in $k$ holds
\begin{equation}
t_{k}\left(  n\right)  -t_{k-1}\left(  n\right)  \longrightarrow0,
\label{tktk1}%
\end{equation}
where $t_{k}=t_{k}\left(  n\right)  $ are the points determined by Theorem
\ref{Thm1}, and for $n\longrightarrow\infty$ uniformly in $k$ holds
\begin{equation}
\frac{\log b\left(  n,k\right)  }{t_{k}-t_{k+1}}\longrightarrow\lambda
_{2}-\lambda_{0}. \label{assumptonb}%
\end{equation}
Then the Bernstein operator $B_{\left(  \lambda_{0},...,\lambda_{n}\right)  }
$ defined in Theorem \ref{Thm1} converges to the identity operator on
$C([a,b],\mathbb{C})$ with the uniform norm.
\end{theorem}

Theorem \ref{Thm1} applied to the system $\Lambda_{2s+1}$ considered in
Section \ref{S2} shows that there exist unique points $t_{0}<t_{1}<...<t_{n}$
in the interval $\left[  a,b\right]  $ and unique \emph{positive} coefficients
$\alpha_{0},...,\alpha_{n}$ such that the operator $B_{n}:C\left[  a,b\right]
\rightarrow E_{\Lambda_{2s+1}}$ defined by (\ref{defBN}) has the property
\[
B_{\Lambda_{2s+1}}\left(  e^{x}\right)  =e^{x}\text{ and }B_{\Lambda_{2s+1}%
}\left(  e^{-x}\right)  =e^{-x}.
\]
The contributions in this paper may serve to investigate the question whether
the Bernstein operator $B_{\Lambda_{2s+1}}$ converges to the identity operator
for $s\rightarrow\infty.$ Since in this example one has only two different
eigenvalues Theorem \ref{ThmCon} has to be modified. Note that for $k=n$ the
coefficient $a\left(  n,k\right)  $ defined in (\ref{assump1}) is equal to
\[
\lim_{x\rightarrow b}\frac{p_{\left(  \lambda_{0},\lambda_{2},....,\lambda
_{n}\right)  ,n}\left(  x\right)  }{p_{\left(  \lambda_{1},\lambda
_{2},....,\lambda_{n}\right)  ,n}\left(  x\right)  }=\frac{\Phi_{2s}\left(
b-a\right)  }{\Phi_{2s-}\left(  b-a\right)  }%
\]
and in (\ref{eqphic}) we have proved that these numbers converge to $1.$

\section{\label{S1}Taylor expansion of the fundamental function}

We say that the vector $\Lambda_{n}\in\mathbb{C}^{n+1}$is \emph{equivalent} to
the vector $\Lambda_{n}^{\prime}\in\mathbb{C}^{n+1}$ if the corresponding
differential operators are equal (so the spaces of all solutions are equal).
This is the same to say that each $\lambda$ occurs in $\Lambda_{n}$ and
$\Lambda_{n}^{\prime}$ with the same multiplicity. Since the differential
operator $L$ defined in (\ref{eqdefL}) does not depend on the order of
differentiation, it is clear that each permutation of the vector $\Lambda_{n}$
is equivalent to $\Lambda_{n}.$ Hence the space $E_{\left(  \lambda
_{0},...,\lambda_{n}\right)  }$ does not depend on the order of the
eigenvalues $\lambda_{0},...,\lambda_{n}.$

We say that the space $E_{\left(  \lambda_{0},...,\lambda_{n}\right)  }$ is
\emph{closed under complex conjugation}, if for $f\in$ $E_{\left(  \lambda
_{0},...,\lambda_{n}\right)  }$ the complex conjugate function $\overline{f}$
is again in $E_{\left(  \lambda_{0},...,\lambda_{n}\right)  }.$ It is easy to
see that for complex numbers $\lambda_{0},...,\lambda_{n}$ the space
$E_{\left(  \lambda_{0},...,\lambda_{n}\right)  }$ is closed under complex
conjugation if and only if there exists a permutation $\sigma$ of the indices
$\left\{  0,...,n\right\}  $ such that $\overline{\lambda_{j}}=\lambda
_{\sigma\left(  j\right)  }$ for $j=0,...,n.$ In other words, $E_{\left(
\lambda_{0},...,\lambda_{n}\right)  }$ is closed under complex conjugation if
and only if the vector $\Lambda_{n}=\left(  \lambda_{0},...,\lambda
_{n}\right)  $ is equivalent to the conjugate vector $\overline{\Lambda_{n}}$.

In the case of pairwise different $\lambda_{j},j=0,...,n,$ the space
$E_{\left(  \lambda_{0},...,\lambda_{n}\right)  }$ is the linear span
generated by the functions
\[
e^{\lambda_{0}x},e^{\lambda_{1}x},....,e^{\lambda_{n}x}.
\]
In the case when some $\lambda_{j}$ occurs $m_{j}$ times in $\Lambda
_{n}=\left(  \lambda_{0},...,\lambda_{n}\right)  $ a basis of the space
$E_{\left(  \lambda_{0},...,\lambda_{n}\right)  }$ is given by the linearly
independent functions
\[
x^{s}e^{\lambda_{j}x}\qquad\text{for }s=0,1,...,m_{j}-1.
\]
In the case that $\lambda_{0}=...=\lambda_{n}$ the space $E_{\left(
\lambda_{0},...,\lambda_{n}\right)  }$ is just the space of all polynomials of
degree $\leq n,$ and we shall refer to this as the \emph{polynomial case}, and
we shall denote the fundamental function by $\Phi_{\text{pol,}n}\left(
x\right)  $. Obviously $\Phi_{\text{pol,}n}\left(  x\right)  $ is of a very
simple form, namely
\begin{equation}
\Phi_{\text{pol,}n}\left(  x\right)  =\frac{1}{n!}x^{n}, \label{Fundpol}%
\end{equation}
so the Taylor expansion is evident from (\ref{Fundpol}). Moreover one has a
very elegant recursion, namely
\[
\Phi_{\text{pol,}n+1}\left(  x\right)  =\frac{1}{n+1}\cdot x\cdot
\Phi_{\text{pol,}n}\left(  x\right)  .
\]
Generally, recursion formulas for the fundamental function $\Phi_{\Lambda_{n}%
}$ are not known or maybe non-existing, except the case that the eigenvalues
are equidistant, cf. \cite{Li85}. We emphasize that the important (and easy to
prove) formula
\begin{equation}
\left(  \frac{d}{dx}-\lambda_{n+1}\right)  \Phi_{\left(  \lambda
_{0},...,\lambda_{n+1}\right)  }\left(  x\right)  =\Phi_{\left(  \lambda
_{0},...,\lambda_{n}\right)  }\left(  x\right)  \label{rec0}%
\end{equation}
is not a recursion formula from which we may compute $\Phi_{\left(
\lambda_{0},...,\lambda_{n+1}\right)  }\left(  x\right)  $ from $\Phi_{\left(
\lambda_{0},...,\lambda_{n}\right)  }\left(  x\right)  .$

Later we need the Taylor expansion of the fundamental function $\Phi
_{\Lambda_{n}}$ which is probably a folklore result and not difficult to
prove; as definition of the fundamental function we take formula (\ref{defPhi}).

\begin{proposition}
\label{PropTaylor}The function $\Phi_{\Lambda_{n}}$ with $\Lambda_{n}=\left(
\lambda_{0},...,\lambda_{n}\right)  $ satisfies $\Phi_{\left(  \lambda
_{0},...,\lambda_{n}\right)  }^{\left(  k\right)  }\left(  0\right)  =0$ for
$k=0,...,n-1$, and for $k\geq n$ the formula
\[
\Phi_{\left(  \lambda_{0},...,\lambda_{n}\right)  }^{\left(  k\right)
}\left(  0\right)  =\sum_{s_{0}+...+s_{n}+n=k}^{\infty}\lambda_{0}^{s_{0}%
}...\lambda_{n}^{s_{n}}.
\]
holds. In particular, $\Phi_{\left(  \lambda_{0},...,\lambda_{n}\right)
}^{\left(  n\right)  }\left(  0\right)  =1$ and $\Phi_{\left(  \lambda
_{0},...,\lambda_{n}\right)  }^{\left(  n+1\right)  }\left(  0\right)
=\lambda_{0}+...+\lambda_{n},$ and
\begin{equation}
\Phi_{\left(  \lambda_{0},...,\lambda_{n}\right)  }^{\left(  n+2\right)
}\left(  0\right)  =\sum_{s_{0}+...+s_{n}=2}^{\infty}\lambda_{0}^{s_{0}%
}...\lambda_{n}^{s_{n}}. \label{phiN2}%
\end{equation}
\end{proposition}

\begin{proof}
Write for $z\in\mathbb{C}$ with $\left|  z\right|  >\left|  \lambda
_{j}\right|  $
\[
\frac{1}{z-\lambda_{j}}=\frac{1}{z}\cdot\frac{1}{1-\frac{\lambda_{j}}{z}}%
=\sum_{s=0}^{\infty}\lambda_{j}^{s}\left(  \frac{1}{z}\right)  ^{s+1}.
\]
Thus we have
\[
\Phi_{\left(  \lambda_{0},...,\lambda_{n}\right)  }\left(  x\right)
=\sum_{s_{0}=0}^{\infty}...\sum_{s_{n}=0}^{\infty}\frac{1}{2\pi i}\int
_{\Gamma_{r}}\lambda_{0}^{s_{0}}...\lambda_{n}^{s_{n}}\frac{e^{xz}}%
{z^{s_{0}+...+s_{n}+n+1}}dz.
\]
By differentiating one obtains
\[
\Phi_{\left(  \lambda_{0},...,\lambda_{n}\right)  }^{\left(  k\right)
}\left(  x\right)  =\sum_{s_{0}=0}^{\infty}...\sum_{s_{n}=0}^{\infty}\frac
{1}{2\pi i}\int_{\Gamma_{r}}\lambda_{0}^{s_{0}}...\lambda_{n}^{s_{n}}%
\frac{z^{k}e^{xz}}{z^{s_{0}+...+s_{n}+n+1}}dz.
\]
For $x=0$ the integral is easy to evaluate and the result is proven.
\end{proof}

Let us specialize the last result:

\begin{proposition}
\label{PropTaylor3}In the case of $\left(  \lambda_{0},...,\lambda_{n}\right)
=\Lambda_{2s+1}\left(  -1,1\right)  $ the following holds
\[
\Phi_{2s+1}^{\left(  2s\right)  }\left(  0\right)  =1\text{ and }\Phi
_{2s+1}^{\left(  2s+1\right)  }\left(  0\right)  =0\text{ and }\Phi
_{2s+1}^{\left(  2s+3\right)  }\left(  0\right)  =s+1
\]
\end{proposition}

\begin{proof}
The equation $\Phi_{2s+1}^{\left(  2s\right)  }\left(  0\right)  =1$ is clear;
further $\Phi_{2s+1}^{\left(  2s+1\right)  }\left(  0\right)  =\lambda
_{0}+...+\lambda_{n}=0.$ Next we use formula (\ref{phiN2}): split up the
integral according to the cases $s_{n}=0,1,2:$ then
\[
\Phi_{\left(  \lambda_{0},...,\lambda_{n}\right)  }^{\left(  n+2\right)
}\left(  0\right)  =\Phi_{\left(  \lambda_{0},...,\lambda_{n-1}\right)
}^{\left(  n+1\right)  }\left(  0\right)  +\sum_{j=0}^{n}\lambda_{j}%
\lambda_{n}=\Phi_{\left(  \lambda_{0},...,\lambda_{n-1}\right)  }^{\left(
n+1\right)  }\left(  0\right)  +\lambda_{n}\Phi_{\left(  \lambda
_{0},...,\lambda_{n}\right)  }^{\left(  n+1\right)  }\left(  0\right)  .
\]
Using this formula for $\Phi_{\left(  \lambda_{0},...,\lambda_{n-1}\right)
}^{\left(  n+1\right)  }\left(  0\right)  $ (instead of $\Phi_{\left(
\lambda_{0},...,\lambda_{n}\right)  }^{\left(  n+2\right)  }\left(  0\right)
)$ one obtains
\[
\Phi_{\left(  \lambda_{0},...,\lambda_{n}\right)  }^{\left(  n+2\right)
}\left(  0\right)  =\Phi_{\left(  \lambda_{0},...,\lambda_{n-2}\right)
}^{\left(  n\right)  }\left(  0\right)  +\lambda_{n-1}\Phi_{\left(
\lambda_{0},...,\lambda_{n-1}\right)  }^{\left(  n\right)  }\left(  0\right)
+\lambda_{n}\Phi_{\left(  \lambda_{0},...,\lambda_{n}\right)  }^{\left(
n+1\right)  }\left(  0\right)  .
\]
Applied to $\Lambda_{n}=\Lambda_{2s+1}\left(  -1,1\right)  $ and $2s+1=n$ we
see that $\Phi_{\left(  \lambda_{0},...,\lambda_{n}\right)  }^{\left(
n+1\right)  }\left(  0\right)  =0$ and $\Phi_{\left(  \lambda_{0}%
,...,\lambda_{n-1}\right)  }^{\left(  n\right)  }\left(  0\right)
=\lambda_{n-1},$ so we have
\[
\Phi_{\left(  \lambda_{0},...,\lambda_{n}\right)  }^{\left(  n+2\right)
}\left(  0\right)  =\Phi_{\left(  \lambda_{0},...,\lambda_{n-2}\right)
}^{\left(  n\right)  }\left(  0\right)  +1.
\]
By iterating one has $\Phi_{\left(  \lambda_{0},...,\lambda_{n}\right)
}^{\left(  n+2\right)  }\left(  0\right)  =\Phi_{\left(  \lambda
_{0},...,\lambda_{n-2j}\right)  }^{\left(  n-2j\right)  }\left(  0\right)
+j.$ Since $n=2s+1$ is odd we can put $j=s$ and obtain
\[
\Phi_{2s+1}^{\left(  2s+3\right)  }\left(  0\right)  =s+\Phi_{\left(
\lambda_{0}\right)  }^{\left(  1\right)  }\left(  0\right)  =s+1
\]
since $\Phi_{\left(  \lambda_{0}\right)  }\left(  x\right)  =e^{\lambda_{0}x}$
and $\lambda_{0}=1.$
\end{proof}

We finish the section by recalling two standard facts:

\begin{proposition}
The function $\Phi_{\Lambda_{n}}$ is real-valued if $\left(  \overline
{\lambda_{0}},...,\overline{\lambda_{n}}\right)  $ is equivalent to $\left(
\lambda_{0},...,\lambda_{n}\right)  .$
\end{proposition}

\begin{proposition}
If $\lambda_{0},...,\lambda_{n}$ are real then $\Phi_{\Lambda_{n}}\left(
x\right)  >0$ for all $x>0.$
\end{proposition}

\begin{proof}
Since $\lambda_{0},...,\lambda_{n}$ are real the space $E_{\left(  \lambda
_{0},...,\lambda_{n}\right)  }$ is a Chebyshev space over $\mathbb{R}$, so
$\Phi_{\Lambda_{n}}$ has at most $n$ zeros in $\mathbb{R}.$ Since
$\Phi_{\Lambda_{n}}$ has exactly $n$ zeros in $0,$ it has no other zeros. By
the norming condition $\Phi_{\Lambda_{n}}^{\left(  n\right)  }\left(
0\right)  =1$ it follows that $\Phi_{\Lambda_{n}}\left(  x\right)  >0$ for
small $x>0,$ hence $\Phi_{\Lambda_{n}}\left(  x\right)  >0$ for all $x>0.$
\end{proof}

The following result is a simple consequence of the definition of
$\Phi_{\Lambda_{n}}$

\begin{proposition}
\label{PropSub}If $\Lambda_{n}=\left(  \lambda_{0},...,\lambda_{n}\right)  $
and $c+\Lambda_{n}:=\left(  c+\lambda_{0},...,c+\lambda_{n}\right)  $ for some
$c\in\mathbb{C}$ then
\[
\Phi_{c+\Lambda_{n}}\left(  x\right)  =e^{cx}\Phi_{\Lambda_{n}}\left(
x\right)  .
\]
If $c\Lambda_{n}=\left(  c\lambda_{0},...,c\lambda_{n}\right)  $ for $c\neq0$
then
\[
\Phi_{c\Lambda_{n}}\left(  x\right)  =\frac{1}{c^{n}}\Phi_{\Lambda_{n}}\left(
cx\right)
\]
\end{proposition}

\section{\label{S2}The fundamental function for $\Lambda_{2s+1}\left(
\lambda,\mu\right)  $}

Let $\Gamma_{r}\left(  t\right)  =re^{it}$ for $t\in\left[  0,2\pi\right]  $
and fixed $r>0.$ Let $\lambda$ and $\mu$ be two real numbers and assume that
$r>0$ be so large that $\lambda$ and $\mu$ are contained in the open ball of
radius $r$ and center $0.$ In the following we want to compute and analyze the
fundamental function
\[
\Phi_{\Lambda_{s}\left(  \lambda,\mu\right)  }\left(  x\right)  =\frac{1}{2\pi
i}\int_{\Gamma_{r}}\frac{e^{xz}}{\left(  z-\lambda\right)  ^{s+1}\left(
z-\mu\right)  ^{s+1}}dz.
\]
By Proposition \ref{PropSub} it is sufficient to consider the case $\lambda=1$
and $\mu=-1,$ so we define
\[
\Phi_{2s+1}\left(  x\right)  :=\frac{1}{2\pi i}\int_{\Gamma_{r}}\frac{e^{xz}%
}{\left(  z-1\right)  ^{s+1}\left(  z+1\right)  ^{s+1}}dz
\]
This integral can be evaluated by the residue theorem, giving two summands
according to the poles $-1$ and $1.$ By a simple substitution argument we see
that
\[
\Phi_{2s+1}\left(  x\right)  =\frac{1}{2\pi i}\int_{\Gamma_{1}}\frac
{e^{x\left(  z+1\right)  }}{z^{s+1}\left(  z+2\right)  ^{s+1}}dz+\frac{1}{2\pi
i}\int_{\Gamma_{1}}\frac{e^{x\left(  z-1\right)  }}{z^{s+1}\left(  z-2\right)
^{s+1}}dz
\]
where $\Gamma_{1}\left(  t\right)  =e^{it}$ for $t\in\left[  0,2\pi\right]  .$
We substitute in the second integral the variable $z$ by $-z$ and obtain
\[
\Phi_{2s+1}\left(  x\right)  =\frac{e^{x}}{2\pi i}\int_{\Gamma_{1}}\frac
{e^{x}}{z^{s+1}\left(  z+2\right)  ^{s+1}}dz-\frac{e^{-x}}{2\pi i}\int
_{\Gamma_{1}}\frac{e^{-xz}}{z^{s+1}\left(  z+2\right)  ^{s+1}}dz.
\]
Now for an integer $\alpha$ we define the polynomials
\begin{equation}
P_{s}^{\alpha}\left(  x\right)  =\frac{1}{2\pi i}\int_{\Gamma_{1}}\frac
{e^{xz}}{z^{s+1}\left(  z+2\right)  ^{s+1+\alpha}}dz \label{Psalfa}%
\end{equation}
and we see that
\begin{equation}
\Phi_{2s+1}\left(  x\right)  =e^{x}P_{s}^{0}\left(  x\right)  -e^{-x}P_{s}%
^{0}\left(  -x\right)  . \label{idphi1}%
\end{equation}
Using equality $\cosh x=\frac{1}{2}\left(  e^{x}+e^{-x}\right)  $ and
$\sinh\left(  x\right)  =\frac{1}{2}\left(  e^{x}-e^{-x}\right)  $ it is easy
to see that
\[
\Phi_{2s+1}\left(  x\right)  =\cosh\left(  x\right)  \left[  P_{s}^{0}\left(
x\right)  -P_{s}^{0}\left(  -x\right)  \right]  +\sinh\left(  x\right)
\left[  P_{s}^{0}\left(  x\right)  +P_{s}^{0}\left(  -x\right)  \right]  .
\]
For example, it is easy to see that
\[
\Phi_{1}\left(  x\right)  =\sinh x\text{ and }\Phi_{3}\left(  x\right)
=\frac{1}{2}\left(  x\cosh x-\sinh x\right)  .
\]
A similar consideration shows that for the function $\Phi_{2s}\left(
x\right)  $ defined as%
\[
\Phi_{2s}\left(  x\right)  :=\frac{1}{2\pi i}\int_{\Gamma_{r}}\frac{e^{xz}%
}{\left(  z-1\right)  ^{s+1}\left(  z+1\right)  ^{s}}dz
\]
we have
\begin{align*}
\Phi_{2s}\left(  x\right)   &  =\frac{e^{x}}{2\pi i}\int_{\Gamma_{1}}%
\frac{e^{xz}}{z^{s+1}\left(  z+2\right)  ^{s}}dz+\frac{e^{-x}}{2\pi i}%
\int_{\Gamma_{1}}\frac{e^{xz}}{z^{s}\left(  z-2\right)  ^{s+1}}dz\\
&  =\frac{e^{x}}{2\pi i}\int_{\Gamma_{1}}\frac{e^{xz}}{z^{s+1}\left(
z+2\right)  ^{s}}dz+\frac{e^{-x}}{2\pi i}\int_{\Gamma_{1}}\frac{e^{-xz}}%
{z^{s}\left(  z+2\right)  ^{s+1}}dz\\
&  =e^{x}P_{s}^{-1}\left(  x\right)  +e^{-x}P_{s-1}^{1}\left(  -x\right)  .
\end{align*}
Here for the case $s=0$ we use the convention $P_{-1}^{\alpha}\left(
x\right)  :=0.$ The following is straightforward:

\begin{lemma}
The functions $P_{s}^{0}\left(  x\right)  $ are polynomials in the variable
$x$ of degree $s$ given by the formula
\[
P_{s}^{0}\left(  x\right)  =\left(  -1\right)  ^{s}\frac{1}{2^{2s+1}}%
\sum_{k=0}^{s}\frac{1}{k!\left(  s-k\right)  !}\frac{\left(  2s-k\right)
!}{s!}\left(  -2x\right)  ^{k}%
\]
\end{lemma}

\begin{proof}
From residue theory it is known that
\[
P_{s}^{0}\left(  x\right)  =\frac{1}{s!}\left[  \frac{d^{s}}{dz^{s}}\left[
e^{xz}\left(  z+2\right)  ^{-s-1}\right]  \right]  _{z=0}.
\]
The rule of Leibniz gives
\[
\frac{d^{s}}{dz^{s}}\left[  e^{xz}\left(  z+2\right)  ^{-s-1}\right]
=\sum_{k=0}^{s}\binom{s}{k}\frac{d^{k}}{dz^{k}}\left(  z+2\right)
^{-s-1}\cdot\frac{d^{s-k}}{dz^{s-k}}e^{xz}.
\]
Note that $\frac{d^{k}}{dz^{k}}\left(  z+2\right)  ^{-s-1}=\left(
-s-1\right)  ...\left(  -s-k\right)  \left(  z+2\right)  ^{-s-1-k},$ and
clearly $\frac{d^{s-k}}{dz^{s-k}}e^{xz}=e^{xz}x^{s-k},$ so
\[
\frac{d^{s}}{dz^{s}}\left[  e^{xz}\left(  z+2\right)  ^{-s-1}\right]
=e^{xz}\sum_{k=0}^{s}\binom{s}{k}\left(  -1\right)  ^{k}\frac{\left(
s+k\right)  !}{s!}x^{s-k}\left(  z+2\right)  ^{-s-1-k}.
\]
Reverse now the summation, and we arrive at
\[
\frac{d^{s}}{dz^{s}}\left[  e^{xz}\left(  z+2\right)  ^{-s-1}\right]
=e^{xz}\sum_{k=0}^{s}\binom{s}{k}\left(  -1\right)  ^{s-k}\frac{\left(
2s-k\right)  !}{s!}x^{k}\left(  z+2\right)  ^{-s-1-\left(  s-k\right)  }.
\]
Take now $z=0$.
\end{proof}

The following is a short list of the first four polynomials for $\alpha=0:$
\begin{align*}
P_{0}^{0}\left(  x\right)   &  =\frac{1}{2}\\
P_{1}^{0}\left(  x\right)   &  =\frac{1}{4}\left(  x-1\right) \\
P_{2}^{0}\left(  x\right)   &  =\frac{1}{2^{5}}\left(  2x^{2}-6x+6\right) \\
P_{3}^{0}\left(  x\right)   &  =\frac{1}{2^{6}}\left(  -20+20x-8x^{2}+\frac
{4}{3}x^{3}\right) \\
P_{4}^{0}\left(  x\right)   &  =\frac{1}{2^{9}}\left(  70-70x+30x^{2}%
-\frac{20}{3}x^{3}+\frac{2}{3}x^{4}\right)
\end{align*}
The polynomial $P_{2}^{0}\left(  x\right)  $ is strictly positive on the real
line, hence the polynomials $P_{s}^{0}\left(  x\right)  $ are not orthogonal
polynomials with respect to any measure on the real line. However the
following is true:

\begin{theorem}
The polynomials $P_{s}^{0}$ satisfy the following recurrence relation:
\begin{equation}
4s\left(  s+1\right)  P_{s+1}^{0}\left(  x\right)  =x^{2}P_{s-1}^{0}\left(
x\right)  -2s\left(  2s+1\right)  P_{s}^{0}\left(  x\right)  . \label{rec2}%
\end{equation}
\end{theorem}

This can be derived by a direct but somewhat tedious calculation. For the
fundamental function $\Phi_{2s+1}=e^{x}P_{s}^{0}\left(  x\right)  -e^{-x}%
P_{s}^{0}\left(  -x\right)  $ we obtain by a straightforward calculation the
recurrence relation (\ref{rec3}) below. Since we shall derive this recurrence
relation from Theorem \ref{ThmR2} in Section \ref{S4} by a different method,
we omit the proof of (\ref{rec2}),

\begin{corollary}
\label{CorFund}The fundamental function $\Phi_{2s+1}\left(  x\right)  $
satisfies the recursion
\begin{equation}
\Phi_{2s+3}\left(  x\right)  =x^{2}\frac{1}{4s\left(  s+1\right)  }\Phi
_{2s-1}\left(  x\right)  -\frac{2s+1}{2s+2}\Phi_{2s+1}\left(  x\right)  ,
\label{rec3}%
\end{equation}
and the following estimate holds for all $x>0:$
\begin{equation}
0\leq\frac{\Phi_{2s+1}\left(  x\right)  }{\Phi_{2s-1}\left(  x\right)  }%
<\frac{x^{2}}{2s\left(  2s+1\right)  }. \label{rec4}%
\end{equation}
\end{corollary}

\begin{proof}
Since $\Phi_{2s+3}\left(  x\right)  >0$ the equation (\ref{rec3}) implies
that
\[
\frac{2s+1}{2s+2}\Phi_{2s+1}\left(  x\right)  <x^{2}\frac{1}{4s\left(
s+1\right)  }\Phi_{2s-1}\left(  x\right)
\]
from which (\ref{rec4}) is immediate.
\end{proof}

Let us recall that after formula (\ref{eqphic}) we defined $\Phi_{2s}\left(
x\right)  $ as the fundamental function for the vector with $s+1$ many $1$ and
$s$ many $-1$. By $\Phi_{2s-}\left(  x\right)  $ we denoted the fundamental
function with respect to the vector with $s+1$ many $-1$ and $s$ many $1.$ We
shall denote sometimes $\Phi_{2s}\left(  x\right)  $ also by $\Phi
_{2s+}\left(  x\right)  $ in order to facilitate some formulas and to
underline the difference to $\Phi_{2s-}\left(  x\right)  .$ The following
simple identity
\[
\frac{1}{\left(  z-1\right)  ^{s+1}\left(  z+1\right)  ^{s}}-\frac{1}{\left(
z-1\right)  ^{s}\left(  z+1\right)  ^{s+1}}=\frac{2}{\left(  z-1\right)
^{s+1}\left(  z+1\right)  ^{s+1}}%
\]
implies the formula
\[
\Phi_{2s}\left(  x\right)  -\Phi_{2s-}\left(  x\right)  =2\Phi_{2s+1}\left(
x\right)  .
\]
This formula can also be derived by summing up the following two identities in
Theorem \ref{Thmneuid} which we shall derive from Theorem \ref{ThmR1}.

\begin{theorem}
\label{Thmneuid}The following two recursions hold:
\begin{align*}
\Phi_{2s+1}\left(  x\right)   &  =\Phi_{2s}\left(  x\right)  -\frac{1}%
{2s}x\cdot\Phi_{2s-1}\left(  x\right)  ,\\
\Phi_{2s+1}\left(  x\right)   &  =-\Phi_{2s-}\left(  x\right)  +\frac{1}%
{2s}x\cdot\Phi_{2s-1}\left(  x\right)  .
\end{align*}
\end{theorem}

\begin{proof}
We derive the result from Theorem \ref{ThmR1}: We choose $k=2s-1$ in the
relation
\[
A_{\pm}p_{2s+1,2s+1}\left(  x\right)  =x\cdot p_{2s-1,2s-1}-2s\cdot
p_{2s\pm,2s}\left(  x\right)
\]
or which is the same (up to notation)
\[
A_{\pm}\Phi_{2s+1}\left(  x\right)  =x\cdot\Phi_{2s-1}-2s\cdot\Phi_{2s\pm
}\left(  x\right)  .
\]
By Proposition \ref{PropTaylor}
\begin{align*}
\Phi_{2s-1}^{2s}\left(  0\right)   &  =\lambda_{0}+....+\lambda_{2s-1}=0\\
\Phi_{2s\pm}^{\left(  2s+1\right)  }  &  =\lambda_{0}+...+\lambda_{2s\pm}%
=\pm1.
\end{align*}
So $A_{\pm}=\left(  2s+1\right)  \Phi_{2s-1}^{\left(  2s\right)  }\left(
0\right)  -2s\Phi_{2s\pm}^{\left(  2s+1\right)  }\left(  0\right)
=\pm1\left(  -2s\right)  .$
\end{proof}

\begin{corollary}
The following limit exists
\[
\lim_{s\rightarrow\infty}\frac{\Phi_{2s}\left(  x\right)  }{\Phi_{2s-}\left(
x\right)  }\rightarrow1.
\]
\end{corollary}

\begin{proof}
By Theorem \ref{Thmneuid} we have
\[
\frac{\Phi_{2s}\left(  x\right)  }{\Phi_{2s-}\left(  x\right)  }=\frac
{\Phi_{2s+1}\left(  x\right)  +\frac{1}{2s}x\cdot\Phi_{2s-1}\left(  x\right)
}{\frac{1}{2s}x\cdot\Phi_{2s-1}\left(  x\right)  -\Phi_{2s+1}\left(  x\right)
}.
\]
Let us define $y_{s}:=\Phi_{2s+1}\left(  x\right)  /\Phi_{2s-1}\left(
x\right)  ,$ then
\begin{equation}
\frac{\Phi_{2s}\left(  x\right)  }{\Phi_{2s-}\left(  x\right)  }=\frac{s\cdot
y_{s}+\frac{1}{2}x}{\frac{1}{2}x-s\cdot y_{s}}. \label{eqschoen}%
\end{equation}
From Theorem \ref{CorFund} we see that $s\cdot y_{s}$ converges to $0,$ so
(\ref{eqschoen}) converges to $1.$
\end{proof}

We mention that one can derive also recursion formula for the derivatives,
e.g. the following identity holds:

\begin{theorem}
The derivatives $\frac{d}{dx}P_{s}^{0}$ of the polynomials $P_{s}^{0}\left(
x\right)  $ can be computed by
\[
\frac{d}{dx}P_{s}^{0}\left(  x\right)  =\frac{1}{s}\frac{x}{2}P_{s-1}%
^{0}\left(  x\right)  -P_{s}^{0}\left(  x\right)  .
\]
\end{theorem}

\section{\label{S3}Generating functions}

The Lagrange inversion formula, see e.g. \cite{AAR}, is another way to
investigate the polynomials defined in (\ref{Psalfa})
\begin{equation}
P_{n}^{\alpha}\left(  x\right)  =\frac{1}{2\pi i}\int_{\Gamma_{1}}\frac
{e^{xz}}{z^{n+1}\left(  z+2\right)  ^{n+1+\alpha}}dz, \label{defpn}%
\end{equation}
where $\alpha$ is a fixed integer. Since this powerful method is somewhat
technical, let us recall the basic facts. In our case we put $\varphi\left(
z\right)  =1/(z+2)$ and $f^{\prime}\left(  z\right)  :=e^{xz}/\left(
z+2\right)  ^{\alpha}.$ The fundamental idea of Lagrange inversion is based on
the observation that
\begin{equation}
a_{n}:=\frac{1}{2\pi i}\int_{\Gamma_{1}}\frac{f^{\prime}\left(  z\right)
}{z^{n+1}}\left[  \varphi\left(  z\right)  \right]  ^{n+1}dz \label{eqan}%
\end{equation}
can be seen as the $n$-th Taylor coefficient of a holomorphic function which
will be constructed from $f^{\prime}$ and $\varphi.$ More generally, we may
assume that $\varphi$ is a holomorphic function in a neighborhood of $0$ such
that $\varphi\left(  0\right)  \neq0,$ and $f$ is holomorphic in a
neighbhorhood of $0,$ and we define $a_{n}$ by the expression (\ref{eqan}).
Consider the function
\[
y\left(  z\right)  :=\frac{z}{\varphi\left(  z\right)  }%
\]
which is holomorphic in a neighborhood of $0$ (since $\varphi\left(  0\right)
\neq0)$ with $y\left(  0\right)  =0.$ Since obviously $y\left(  z\right)
\varphi\left(  z\right)  =z$ we obtain $\varphi\left(  z\right)  y^{\prime
}\left(  z\right)  +y\left(  z\right)  \varphi^{\prime}\left(  z\right)  =1,$
so $y^{\prime}\left(  0\right)  \neq0.$ Hence $y$ is injective in a
neighborhood; let $y^{-1}$ be the inverse map. Since $y^{-1}\circ y\left(
z\right)  =z$ one has $\frac{d}{dz}y^{-1}\left(  y\left(  z\right)  \right)
\cdot\frac{d}{dz}y\left(  z\right)  =1,$ and using this formula one arrives
at
\[
(\frac{d}{dy}\left(  f\circ y^{-1}\right)  )\left(  y\right)  =f^{\prime
}\left(  y^{-1}\left(  y\right)  \right)  \cdot(\frac{d}{dz}y^{-1})\left(
y\right)  =f^{\prime}\left(  z\right)  \cdot\frac{1}{\frac{d}{dz}y\left(
z\right)  }.
\]
for $y=y\left(  z\right)  .$ Thus we obtain
\begin{equation}
\frac{f^{\prime}\left(  z\right)  \varphi\left(  z\right)  ^{n+1}}{z^{n+1}%
}=\frac{f^{\prime}\left(  z\right)  }{y\left(  z\right)  ^{n+1}}=\frac
{\frac{d}{dy}\left(  f\circ y^{-1}\right)  \left(  y\left(  z\right)  \right)
\cdot\frac{d}{dz}y\left(  z\right)  }{y\left(  z\right)  ^{n+1}}.
\label{eqresid}%
\end{equation}
Let now $\gamma\left(  t\right)  =re^{it}$ for $r>0$ sufficient small, and put
$z=\gamma\left(  t\right)  .$ Note that $\Gamma\left(  t\right)  =y\left(
\gamma\left(  t\right)  \right)  $ is a path surrounding zero. Insert
$z=\gamma\left(  t\right)  $ in (\ref{eqresid}) and multiply it with
$\gamma^{\prime}\left(  t\right)  /2\pi i$, so we obtain
\begin{align*}
a_{n}  &  =\frac{1}{2\pi i}\int_{\gamma}\frac{f^{\prime}\left(  z\right)
}{z^{n+1}}\left[  \varphi\left(  z\right)  \right]  ^{n+1}dz\\
&  =\frac{1}{2\pi i}\int_{0}^{2\pi}\frac{\frac{d}{dy}\left(  f\circ
y^{-1}\right)  \left(  y\left(  \gamma\left(  t\right)  \right)  \right)
\cdot\frac{d}{dz}y\left(  \gamma\left(  t\right)  \right)  }{y\left(
\gamma\left(  t\right)  \right)  ^{n+1}}\gamma^{\prime}\left(  t\right)  dt\\
&  =\frac{1}{2\pi i}\int_{\Gamma}\frac{\frac{d}{dy}\left(  f\circ
y^{-1}\right)  \left(  y\right)  }{y^{n+1}}dy.
\end{align*}
Let us define $F\left(  y\right)  =\frac{d}{dy}\left(  f\circ y^{-1}\right)
\left(  y\right)  ,$ then by the Cauchy formula for the Taylor coefficients
one obtains
\[
a_{n}=\frac{1}{n!}F^{\left(  n\right)  }\left(  0\right)  =\frac{1}{n!}%
\frac{d^{n}}{dy^{n}}F\left(  0\right)  .
\]
Since the Taylor series of $F\left(  y\right)  $ is equal to $F\left(
y\right)  $ we obtain
\begin{equation}
\sum_{n=0}^{\infty}a_{n}y^{n}=F\left(  y\right)  =\frac{d}{dy}\left(  f\circ
y^{-1}\right)  \left(  y\right)  =f^{\prime}\left(  y^{-1}\left(  y\right)
\right)  \frac{dy^{-1}}{dy}\left(  y\right)  . \label{eqLagrange}%
\end{equation}
So this means that for the computation of the unknown sum $\sum_{n=0}^{\infty
}a_{n}y^{n}$ we only have to compute $y^{-1}\left(  y\right)  $, and the
derivative $\frac{dy^{-1}}{dy}\left(  y\right)  $ and then we have to compute
the right hand side of (\ref{eqLagrange}). As an application we prove:

\begin{theorem}
\label{ThmGener}Let $P_{n}^{\alpha}\left(  x\right)  $ be defined by the
expression (\ref{defpn}). Then for all $y$ with $\left|  y\right|  <1$
\[
F_{P^{\alpha}}\left(  x,y\right)  :=\sum_{n=0}^{\infty}P_{n}^{\alpha}\left(
x\right)  y^{n}=e^{-x}\frac{e^{x\sqrt{y+1}}}{2\sqrt{y+1}}\frac{1}{\left(
\sqrt{y+1}+1\right)  ^{\alpha}}.
\]
\end{theorem}

\begin{proof}
Let $\varphi\left(  z\right)  =\left(  z+2\right)  ^{-1}$ and $y\left(
z\right)  :=z/\varphi\left(  z\right)  $, so $y\left(  z\right)  =z\left(
z+2\right)  .$ We define the inverse function $y^{-1}$ of $y\left(  z\right)
$ by
\[
y^{-1}\left(  y\right)  :=\sqrt{y+1}-1.
\]
Then $y^{-1}\left(  y\left(  z\right)  \right)  =\sqrt{z\left(  z+2\right)
+1}-1=z $ and note that
\[
\frac{dy^{-1}}{dy}\left(  y\right)  =\frac{1}{2\sqrt{y+1}}.
\]
Define $f^{\prime}\left(  z\right)  =e^{xz}/\left(  z+2\right)  ^{\alpha}.$
Then the Lagrange inversion formula tells us that
\[
\sum_{n=0}^{\infty}P_{n}^{\alpha}\left(  x\right)  y^{n}=f^{\prime}\left(
\sqrt{y+1}-1\right)  \frac{1}{2\sqrt{y+1}}\frac{1}{\left(  \sqrt
{y+1}+1\right)  ^{\alpha}}%
\]
which is exactly our claim.
\end{proof}

\begin{remark}
It is easy to derive a second order differential equation for the generating
function $y\longmapsto F_{P^{0}}\left(  x,y\right)  ,$ and this can be used to
give a proof for the recursion (\ref{rec2}).
\end{remark}

\begin{theorem}
The following identity
\[
F_{\text{odd }}\left(  x,y\right)  =\sum_{s=0}^{\infty}\Phi_{2s+1}\left(
x\right)  \cdot y^{s}=\frac{1}{\sqrt{y+1}}\sinh\left(  x\sqrt{y+1}\right)
\]
holds for all $y\in\mathbb{C}$.
\end{theorem}

\begin{proof}
Since $\Phi_{2s+1}\left(  x\right)  =e^{x}P_{s}^{0}\left(  x\right)
-e^{-x}P_{s}^{0}\left(  -x\right)  $ we obtain
\[
F_{\text{odd }}\left(  x,y\right)  =e^{x}\sum_{s=0}^{\infty}P_{s}^{0}\left(
x\right)  y^{s}-e^{-x}\sum_{s=0}^{\infty}P_{s}^{0}\left(  -x\right)  y^{s}.
\]
Using Theorem \ref{ThmGener} a short computation shows that this is equal to
\[
F_{\text{odd }}\left(  x,y\right)  =\frac{1}{\sqrt{y+1}}\frac{1}{2}\left(
e^{x\sqrt{y+1}}-e^{-x\sqrt{y+1}}\right)  =\frac{1}{\sqrt{y+1}}\sinh
x\sqrt{y+1}.
\]
Since
\[
\sum_{n=0}^{\infty}P_{n}^{0}\left(  x\right)  y^{n}=\exp\left(  x\left(
\sqrt{y+1}-1\right)  \right)  \frac{1}{2\sqrt{y+1}}%
\]
we obtain
\begin{align*}
F_{\text{odd }}\left(  x,y\right)   &  =\left(  x\sqrt{y+1}\right)  \frac
{1}{2\sqrt{y+1}}-\exp\left(  -x\sqrt{y+1}\right)  \frac{1}{2\sqrt{y+1}}\\
&  =\frac{1}{\sqrt{y+1}}\sinh\left(  x\sqrt{y+1}\right)  .
\end{align*}
Since $\sinh z$ contains only odd powers in the Taylor expansion it follows
that $y\longmapsto F_{\text{odd }}\left(  x,y\right)  $ is entire.
\end{proof}

Let $\lambda_{0},\lambda_{1},....$ be a sequence of real numbers. In the
following we consider the generating function
\[
F_{\Lambda}\left(  x,y\right)  :=\sum_{n=0}^{\infty}\Phi_{\left(  \lambda
_{0},...,\lambda_{n}\right)  }\left(  x\right)  \cdot y^{n}%
\]
for the case that $\lambda_{2j}=1$ and $\lambda_{2j+1}=-1$ for all
$j\in\mathbb{N}_{0}$. With the notation from the introduction we have
\[
\Phi_{\left(  \lambda_{0},...,\lambda_{2s}\right)  }=\Phi_{2s+1}\text{ and
}\Phi_{\left(  \lambda_{0},...,\lambda_{2s-1}\right)  }=\Phi_{2s}%
\]
and now we consider
\begin{equation}
F_{\Lambda}\left(  x,y\right)  =\sum_{s=0}^{\infty}\Phi_{2s+1}\left(
x\right)  \cdot y^{2s+1}+\sum_{s=0}^{\infty}\Phi_{2s}\left(  x\right)  \cdot
y^{2s}. \label{defgen}%
\end{equation}
It is easy to see that $\left(  \frac{d}{dx}+1\right)  \Phi_{2s+1}\left(
x\right)  =\Phi_{2s}\left(  x\right)  $ (cf. the general formula
(\ref{rec0})), so we have
\[
\sum_{s=0}^{\infty}\Phi_{2s}\left(  x\right)  \cdot y^{2s}=\left(  \frac
{d}{dx}+1\right)  \sum_{s=0}^{\infty}\Phi_{2s+1}\left(  x\right)  \cdot
y^{2s}=\left(  \frac{d}{dx}+1\right)  F_{\text{odd }}\left(  x,y^{2}\right)
.
\]
Thus we have
\[
F_{\Lambda}\left(  x,y\right)  =\left(  1+y\right)  F_{\text{odd }}\left(
x,y^{2}\right)  +\frac{d}{dx}F_{\text{odd }}\left(  x,y^{2}\right)
\]
and the following is proved:

\begin{theorem}
With the above notations, the generating function $F_{\Lambda}\left(
x,y\right)  $ in (\ref{defgen}) is equal to
\[
F_{\Lambda}\left(  x,y\right)  =\frac{1+y}{\sqrt{y^{2}+1}}\sinh\left(
x\sqrt{y^{2}+1}\right)  +\cosh\left(  x\sqrt{y^{2}+1}\right)  .
\]
\end{theorem}

\section{\label{S4}Construction of Bernstein bases}

In this Section we return to the general theory of exponential polynomials
where the eigenvalues $\lambda_{0},...,\lambda_{n}$ may be complex numbers. We
shall characterize the spaces $E_{\left(  \lambda_{0},...,\lambda_{n}\right)
}$ of exponential polynomials which admit a Bernstein basis. Let us emphasize
that the existence results already follow from those in \cite{CMP04},
\cite{GoMa03}, \cite{Mazu99}, \cite{Mazu05} in the more general context of
Chebyshev spaces. We follow here our exposition in \cite{AKR} which is based
on a recursive definition of the Bernstein basis, and it seems that this
approach is be different from those in the above cited literature. Further
references on properties of Bernstein-like bases are \cite{FaGo96} and
\cite{Pena02}.

Let us recall some terminology and notations: The $k$-th derivative of a
function $f$ is denoted by $f^{\left(  k\right)  }.$ A function $f\in
C^{n}\left(  I,\mathbb{C}\right)  $ has a \emph{zero of order }$k$
or\emph{\ of multiplicity} $k$ at a point $a\in I$ if $f\left(  a\right)
=...=f^{\left(  k-1\right)  }\left(  a\right)  =0$ and $f^{\left(  k\right)
}\left(  a\right)  \neq0.$ We shall repeatedly use the fact that
\begin{equation}
k!\cdot\lim_{x\rightarrow0}\frac{f\left(  x\right)  }{\left(  x-a\right)
^{k}}=f^{\left(  k\right)  }\left(  a\right)  . \label{eqLim}%
\end{equation}
for any function $f\in C^{\left(  k\right)  }(I)$ with $f\left(  a\right)
=...=f^{\left(  k-1\right)  }\left(  a\right)  =0.$

Let us recall the definition of a Bernstein basis:

\begin{definition}
A system of functions $p_{n,k},$ $k=0,...,n$ in the space $E_{\left(
\lambda_{0},...,\lambda_{n}\right)  }$ is called \emph{Bernstein-like basis}
for $E_{\left(  \lambda_{0},...,\lambda_{n}\right)  }$ and $a\neq b$ if and
only if each function $p_{n,k}$ has a zero of exact order $k$ at $a$ and a
zero of exact order $n-k$ at $b$.
\end{definition}

In the following we want to characterize those spaces $E_{\left(  \lambda
_{0},...,\lambda_{n}\right)  }$ which admit a Bernstein basis. For this, we
need the following simple

\begin{lemma}
\label{Lem1}Suppose that $f_{0},...,f_{n}\in E_{\left(  \lambda_{0}%
,...,\lambda_{n}\right)  }$ have the property that $f_{k}$ has a zero of order
$k$ at $a$ for $k=0,...,n.$ Then $f_{0},...,f_{n}$ is a basis for any interval
$\left[  a,b\right]  .$
\end{lemma}

\begin{proof}
It suffices to show that $f_{0},...,f_{n}$ are linearly independent since
$E_{\left(  \lambda_{0},...,\lambda_{n}\right)  }$ has dimension $n+1.$
Suppose that there exist complex numbers $c_{0},...,c_{n}$ such that
\[
c_{0}f_{0}\left(  x\right)  +....+c_{n}f_{n}\left(  x\right)  =0
\]
for all $x\in\left[  a,b\right]  .$ Since $f_{0}\left(  a\right)  \neq0$ and
$f_{j}\left(  a\right)  =0$ for all $j\geq1$ we obtain $c_{0}=0.$ Next we take
the derivative and obtain the equation
\[
c_{1}f_{1}^{\prime}\left(  x\right)  +....+c_{n}f_{n}^{\prime}\left(
x\right)  =0.
\]
We insert $x=a$ and obtain that $c_{1}=0$ since $f_{1}^{\prime}\left(
a\right)  \neq0$ and $f_{j}^{\prime}\left(  a\right)  =0$ for all $j=2,...,n.$
Now one proceeds inductively.
\end{proof}

It follows from the Lemma \ref{Lem1} that a Bernstein basis $p_{n,k}%
,k=0,...,n$ is necessarily a basis for $E_{\left(  \lambda_{0},...,\lambda
_{n}\right)  }.$ Next we want to show that the basis functions are unique up
to a factor (provided there exists such a basis).

\begin{proposition}
Suppose that there exists a Bernstein basis $p_{n,k},$ $k=0,...,n$ for the
space $E_{\left(  \lambda_{0},...,\lambda_{n}\right)  }$ and $a,b$. If $f\in
E_{\left(  \lambda_{0},...,\lambda_{n}\right)  }$ has a zero of order at least
$k_{0}$ in $a$ and of order at least $n-k_{0}$ at $b$ then there exists a
complex number $c$ such that
\[
f=c\cdot p_{n,k_{0}}.
\]
\end{proposition}

\begin{proof}
Since $p_{n,k},$ $k=0,...,n$ is a basis we can find complex numbers
$c_{0},...,c_{n}$ such that
\[
f=c_{0}p_{n,0}+...+c_{n}p_{n,n}.
\]
We know that $f\left(  a\right)  =...=f^{\left(  k_{0}-1\right)  }\left(
a\right)  =0$ since $f$ has an order of exact order $k_{0}$ in $a.$ So we see
that $0=f\left(  a\right)  =c_{0}p_{n,0}\left(  a\right)  .$ Since
$p_{n,0}\left(  a\right)  \neq0$ we obtain $c_{0}=0.$ We proceed inductively
and obtain that
\[
f=c_{k_{0}}p_{n,k_{0}}+....+c_{n}p_{n,n}.
\]
Now we use the zeros at $b:$ Inserting $x=b$ yields
\[
0=f\left(  b\right)  =c_{n}p_{n,n}\left(  b\right)  .
\]
Since $p_{n,n}\left(  b\right)  \neq0$ (here we use the exact order at the
point $b$) it follows that $c_{n}=0.$ Proceeding inductively one obtains
$f=c_{k_{0}}p_{n,k_{0}}.$
\end{proof}

The proof actually shows the following:

\begin{corollary}
Suppose that there exists a Bernstein basis $p_{n,k},$ $k=0,...,n$ for the
space $E_{\left(  \lambda_{0},...,\lambda_{n}\right)  }$ and $a,b$. If $f\in
E_{\left(  \lambda_{0},...,\lambda_{n}\right)  }$ has more than $n$ zeros in
$\left\{  a,b\right\}  $ then $f=0.$
\end{corollary}

Let us recall the following definition:

\begin{definition}
The space $E_{\left(  \lambda_{0},...,\lambda_{n}\right)  }$ is a Chebyshev
system with respect to the set $A\subset\mathbb{C}$ if each function $f\in
E_{\left(  \lambda_{0},...,\lambda_{n}\right)  }$ that has more than $n$ zeros
in $A$ (including multiplicities) is $0.$
\end{definition}

So we have seen that a necessary condition for the existence of a Bernstein
basis on the interval $\left[  a,b\right]  $ is the property that $E_{\left(
\lambda_{0},...,\lambda_{n}\right)  }$ is a Chebyshev system for the
\emph{set} $\left\{  a,b\right\}  .$ In the next Proposition we shall show
that this property is actually equivalent to the existence.

\begin{proposition}
\label{PropExist}The space $E_{\left(  \lambda_{0},...,\lambda_{n}\right)  }$
is an extended Chebyshev system with respect to $a,b$ if and only if there
exists a (unique) Bernstein-like basis $p_{n,k}:k=0,...,n$, for $a,b$
satisfying the condition
\begin{equation}
k!\lim_{x\rightarrow a,x>a}\frac{p_{n,k}\left(  x\right)  }{\left(
x-a\right)  ^{k}}=p_{n,k}^{\left(  k\right)  }\left(  a\right)  =1.
\label{eqnorm}%
\end{equation}
The Bernstein-like basis functions $p_{n,k}\left(  x\right)  $ are recursively
defined by equations (\ref{eqBP1}), (\ref{eqBP2}), (\ref{eqq}) and
(\ref{eqBP3}) below.
\end{proposition}

\begin{proof}
The necessity was already proved. Assume now that $E_{\left(  \lambda
_{0},...,\lambda_{n}\right)  }$ is an extended Chebyshev system for $a,b.$ It
is convenient to use the following notation:
\begin{equation}
q_{k}\left(  x\right)  :=p_{n,n-k}\left(  x\right)  , \label{eqBP1}%
\end{equation}
so $q_{k}$ has a zero of order $n-k$ at $a$ and a zero of order $k$ at $b$.
Define first
\begin{equation}
q_{0}\left(  x\right)  :=\Phi_{\Lambda_{n}}\left(  x-a\right)  , \label{eqBP2}%
\end{equation}
which clearly has a zero of order $n$ at $a$ and of order at least $0$ at $b.
$ Since $q_{0}$ can not have more than $n$ zeros on $a,b$ by our assumption we
infer that $q_{0}$ has a zero of order $0$ at $b.$ Define $q_{1}%
:=q_{0}^{\left(  1\right)  }-\alpha_{0}q_{0}$ for $\alpha_{0}=q_{0}^{\left(
1\right)  }\left(  b\right)  /q_{0}\left(  b\right)  $ which has a zero of
order at least $n-1$ at $a$ and a zero of order at least $1$ at $b.$ Again our
assumption implies that $q_{1}$ has a zero of order at $n-1 $ at $a$ and a
zero of order $1$ at $b.$ For $k\geq2$ we define $q_{k}$ recursively by
\begin{equation}
q_{k}:=q_{k-1}^{\left(  1\right)  }-\left(  \alpha_{k-1}-\alpha_{k-2}\right)
\cdot q_{k-1}-\beta_{k}q_{k-2} \label{eqq}%
\end{equation}
with coefficients $\alpha_{k-1}-\alpha_{k-2}$ and $\beta_{k}$ to be
determined. By construction we know already that $q_{k-1}$ and $q_{k-2}$
respectively have a zero of order $k-1$ and $k-2$ at $b$, and a zero of order
$n-k+1$ and $n-k+2$ at $a$. So it is clear that $q_{k}$ has a zero of order at
least $k-2$ at $b,$ and a zero of order at least $n-k$ at $a.$ The
coefficients $\alpha_{k-1}-\alpha_{k-2}$ and $\beta_{k}$ are chosen in such a
way that $q_{k}$ will have a zero of order at least $k$ in $b.$ This is
achieved by defining
\begin{equation}
\beta_{k}:=\frac{q_{k-1}^{\left(  k-1\right)  }\left(  b\right)  }%
{q_{k-2}^{\left(  k-2\right)  }\left(  b\right)  }\text{ and }\alpha
_{k-1}:=\frac{q_{k-1}^{\left(  k\right)  }\left(  b\right)  }{q_{k-1}^{\left(
k-1\right)  }\left(  b\right)  }. \label{eqBP3}%
\end{equation}
Again our assumption implies that $q_{k}$ has a zero of order $k$ at $b$ and a
zero of order $n-k$ at $a.$ The condition (\ref{eqnorm}) is easily checked
using (\ref{eqLim}), and (\ref{eqq}) together with induction. The uniqueness
property follows from the above remarks.
\end{proof}

In the rest of the paper we shall call the Bernstein-like basis provided by
Proposition \ref{PropExist} \emph{the Bernstein basis with respect to} $a,b.$

Next we shall give a construction of the Bernstein basis $p_{n,n-k}\left(
x\right)  ,$ $k=0,...,n,$ will is similar to constructions known from the
theory of Chebyshev spaces.

\begin{theorem}
Let $\left(  \lambda_{0},...,\lambda_{n}\right)  \in\mathbb{C}^{n+1}$ and
define for each $k=0,...,n$ the $\left(  k+1\right)  \times\left(  k+1\right)
$ matrix $A_{n,k}\left(  x\right)  $ by
\begin{equation}
A_{n,k}\left(  x\right)  :=\left(
\begin{array}
[c]{ccc}%
\Phi_{\Lambda_{n}}\left(  x\right)  & ... & \Phi_{\Lambda_{n}}^{\left(
k\right)  }\left(  x\right) \\
\vdots &  & \vdots\\
\Phi_{\Lambda_{n}}^{\left(  k\right)  }\left(  x\right)  & ... & \Phi
_{\Lambda_{n}}^{\left(  2k\right)  }\left(  x\right)
\end{array}
\right)  . \label{defAnk}%
\end{equation}
Then the matrices $A_{n,k}\left(  b-a\right)  $ are invertible for $k=0,...,n$
if and only if $E_{\left(  \lambda_{0},...,\lambda_{n}\right)  }$ is an
extended Chebyshev system with respect to $\left\{  a,b\right\}  .$
\end{theorem}

\begin{proof}
Assume that the matrices $A_{n,k}\left(  b-a\right)  $ are invertible for
$k=0,...,n.$ It suffices to show that there exists a Bernstein basis with
respect to $a,b.$ For a polynomial $r_{\Lambda_{n},k}\left(  z\right)
=r_{k}z^{k}+r_{k-1}z^{k-1}+...+r_{0}$ with coefficients $r_{0},...,r_{k}%
\in\mathbb{C}$ and $r_{k}=1$ let us define
\begin{equation}
f_{n,k}\left(  x\right)  :=\frac{1}{2\pi i}\int\frac{r_{\Lambda_{n},k}\left(
z\right)  e^{\left(  x-a\right)  z}}{\left(  z-\lambda_{0}\right)  ...\left(
z-\lambda_{n}\right)  }dz. \label{deffnk}%
\end{equation}
It is easy to see that $f_{n,k}$ is in $E_{\left(  \lambda_{0},...,\lambda
_{n}\right)  }$ since
\[
f_{n,k}\left(  x\right)  =\sum_{j=0}^{k}r_{j}\Phi_{\Lambda_{n}}^{\left(
j\right)  }\left(  x-a\right)
\]
From this representation it follows that $f_{n,k}$ has of order at $n-k$ at
$a,$ and that $f^{\left(  n-k\right)  }\left(  a\right)  =r_{k}\Phi
_{\Lambda_{n}}^{\left(  n\right)  }\left(  x-a\right)  =1.$ We want to choose
the coefficients $r_{0},...,r_{k-1}$ such that
\[
\left(  \frac{d^{l}}{dx^{l}}f_{n,k}\right)  \left(  b\right)  =0\text{ for all
}l=0,...,k-1
\]
and that $\left(  \frac{d^{k}}{dx^{k}}f_{n,k}\right)  \left(  b\right)
\neq0.$ Writing down these equations for $l=0,...,k-1$ in matrix form shows
that for $c=b-a$
\[
\left(
\begin{array}
[c]{ccc}%
\Phi_{\Lambda_{n}}\left(  c\right)  & ... & \Phi_{\Lambda_{n}}^{\left(
k-1\right)  }\left(  c\right) \\
\vdots &  & \vdots\\
\Phi_{\Lambda_{n}}^{\left(  k-1\right)  }\left(  c\right)  & ... &
\Phi_{\Lambda_{n}}^{\left(  2k-2\right)  }\left(  c\right)
\end{array}
\right)  \left(
\begin{array}
[c]{c}%
r_{0}\\
\vdots\\
r_{k-1}%
\end{array}
\right)  =-\left(
\begin{array}
[c]{c}%
\Phi_{\Lambda_{n}}^{\left(  k\right)  }\left(  c\right) \\
\vdots\\
\Phi_{\Lambda_{n}}^{\left(  2k-2\right)  }\left(  c\right)
\end{array}
\right)
\]
Since the matrix $A_{n,k-1}\left(  b-a\right)  $ is invertible we can find
clearly $r_{0},...,r_{k-1}\in\mathbb{C}$ which solves the equation. Hence
$f_{n,k}$ has a zero of order at $n-k$ at $a$ and a zero of order at least $k$
in $b.$ Suppose now that $f\left(  b\right)  =....=f_{n,k}^{\left(  k\right)
}\left(  b\right)  =0.$ Then these equations say that
\[
\left(
\begin{array}
[c]{ccc}%
\Phi_{\Lambda_{n}}\left(  c\right)  & ... & \Phi_{\Lambda_{n}}^{\left(
k\right)  }\left(  c\right) \\
\vdots &  & \vdots\\
\Phi_{\Lambda_{n}}^{\left(  k\right)  }\left(  c\right)  & ... & \Phi
_{\Lambda_{n}}^{\left(  2k\right)  }\left(  c\right)
\end{array}
\right)  \left(
\begin{array}
[c]{c}%
r_{0}\\
\vdots\\
r_{k}%
\end{array}
\right)  =0.
\]
Since the matrix $A_{n,k}\left(  b-a\right)  $ is invertible it follows that
$r_{0}=....=r_{k}=0.$ This is a contradiction to the choice $r_{k}=1.$

Now assume that $E_{\left(  \lambda_{0},...,\lambda_{n}\right)  }$ is an
extended Chebyshev system with respect to $\left\{  a,b\right\}  $. Suppose
that the matrix $A_{n,k}\left(  b-a\right)  $ is not invertible for some
$k\in\left\{  0,...,n\right\}  .$ Then we can find $s=(s_{0},...,s_{k})\neq0$
such that $A_{n,k}\left(  b-a\right)  s=0.$ Define $f_{n,k}\left(  x\right)
=\sum_{j=0}^{k}s_{j}\Phi_{\Lambda_{n}}^{\left(  j\right)  }\left(  x-a\right)
.$ Then clearly $f_{n,k}$ has a zero of order at least $n-k$ at $a.$ Further
the equation $A_{n,k}\left(  b-a\right)  s=0$ implies that $f_{n,k}^{\left(
j\right)  }\left(  b\right)  =0$ for $j=0,...,k.$ So $f_{n,k}$ has a zero of
order $n+1$ in the set $\left\{  a,b\right\}  ,$ and our assumption implies
that $f_{n,k}=0.$ By Lemma \ref{Lem1} the system $\Phi_{\Lambda_{n}}%
,....,\Phi_{\Lambda_{n}}^{\left(  n\right)  }$ is a basis of $E_{\left(
\lambda_{0},...,\lambda_{n}\right)  }.$ It follows that $s_{0}=...=s_{n}=0,$ a contradiction.
\end{proof}

The proof of the last theorem actually shows:

\begin{theorem}
Assume that the matrices $A_{n,k}\left(  b-a\right)  $ are invertible for
$k=1,...,n+1$, and let $p_{n,n-k},k=0,...,n$ be the Bernstein basis with
respect to $a,b.$ Then for each $k=0,...,n$ there exists a polynomial
$r_{\Lambda_{n},k}\left(  z\right)  $ of degree $k$ and leading coefficient
$1$ such that
\[
p_{n,n-k}\left(  x\right)  =\frac{1}{2\pi i}\int\frac{r_{\Lambda_{n},k}\left(
z\right)  e^{\left(  x-a\right)  z}}{\left(  z-\lambda_{0}\right)  ...\left(
z-\lambda_{n}\right)  }dz
\]
for all $x\in\mathbb{R}.$
\end{theorem}

For a given vector $\Lambda_{n}=\left(  \lambda_{0},...,\lambda_{n}\right)  $
the matrices $A_{n,k}\left(  z\right)  ,k=0,...,n$ are defined in
(\ref{defAnk}) and we set
\[
Z_{\Lambda_{n}}:=\bigcup_{k=0}^{n}\left\{  z\in\mathbb{C}:\det A_{n,k}\left(
z\right)  =0\right\}  .
\]
Note that $Z_{\Lambda_{n}}$ is a discrete subset of $\mathbb{C}$ since
$z\longmapsto\det A_{n,k}\left(  z\right)  $ is obviously an entire function.
It follows that for given $a\in\mathbb{R}$ the space $E_{\left(  \lambda
_{0},...,\lambda_{n}\right)  }$ is an extended Chebyshev system for $\left\{
a,b\right\}  $ for all $b\in\mathbb{R}$ except a countable discrete subset of
$\mathbb{R}.$ We emphasize that this does not imply that $E_{\left(
\lambda_{0},...,\lambda_{n}\right)  }$ is an extended Chebyshev system for the
\emph{interval} $\left[  a,b\right]  .$ In order to have nice properties of
the basic function $p_{n,k}$ one need the assumption that $E_{\left(
\lambda_{0},...,\lambda_{n}\right)  }$ is an extended Chebyshev system for the
\emph{interval} $\left[  a,b\right]  .$ Indeed, the following result is well
known (at least in the polynomial case):

\begin{theorem}
Suppose that $E_{\left(  \lambda_{0},...,\lambda_{n}\right)  }$ is an extended
Chebyshev system for the interval $\left[  a,b\right]  $ with Bernstein basis
functions $p_{n,k},k=0,...,n.$ Assume that $E_{\left(  \lambda_{0}%
,...,\lambda_{n}\right)  }$ is closed under complex conjugation. Then the
basis functions $p_{n,k}\left(  x\right)  $ are strictly positive on $\left(
a,b\right)  $ for each $k=0,...,n.$ For each $k=1,...,n-1$ there exists a
unique $a<t_{k}<b$ such that
\begin{align*}
&  p_{n,k}\text{ is increasing on }\left[  a,t_{k}\right] \\
&  p_{n,k}\text{ is decreasing on }\left[  t_{k},b\right]  .
\end{align*}
So $p_{n,k}$ has exactly one relative maximum. The function $p_{n,n}$ is
either increasing or there exists $t_{0}\in\left(  a,b\right)  $ such that
$p_{n,n}$ is increasing on $\left[  a,t_{0}\right]  $ and decreasing on
$\left[  t_{0},b\right]  .$
\end{theorem}

\begin{proof}
1. Consider $f_{n,k}\left(  x\right)  =p_{n,k}\left(  x\right)  -\overline
{p_{n,k}\left(  x\right)  }$ for $k=0,...,n.$ Then $f_{n,k}$ has a zero of
order at least $k+1$ in $a$ and a zero of order at least $n-k$ in $b.$
Moerover $f\in E_{\left(  \lambda_{0},...,\lambda_{n}\right)  }$ since
$E_{\left(  \lambda_{0},...,\lambda_{n}\right)  }$ is closed under complex
conjugation. As $E_{\left(  \lambda_{0},...,\lambda_{n}\right)  }$ is an
extended Chebyshev system for $\left\{  a,b\right\}  $ we infer $f_{n,k}=0$,
hence $p_{n,k}$ is real-valued.

2. Since $E_{\left(  \lambda_{0},...,\lambda_{n}\right)  }$ is an extended
Chebyshev system for the interval $\left[  a,b\right]  $ it follows that
$p_{n,k}$ has no zeros in the open interval $\left(  a,b\right)  .$ By the
norming condition $p_{n,k}^{\left(  k\right)  }\left(  a\right)  =1$ it
follows that $p_{n,k}$ is positive on $\left(  a,b\right)  .$ Since the
derivative $p_{n,k}^{\prime}$ is again in $E_{\left(  \lambda_{0}%
,...,\lambda_{n}\right)  }$ it follows that $p_{n,k}^{\prime}$ has at most $n$
zeros on $\left[  a,b\right]  .$

3. We assume that $1\leq k\leq n-1.$ Then $p_{n,k}^{\prime}$ has a zero order
$k-1\geq0$ at $a$ and a zero of order $n-k-1$ at $b.$ Hence $p_{n,k}^{\prime}$
has at most $2$ zeros in the open interval $\left(  a,b\right)  .$ First
assume that $p_{n,k}^{\prime}$ has two different zeros $t_{0}$ and $t_{1}$.
Then they must be simple. It follows that $p_{n,k}^{\prime}\left(  t\right)
>0$ for $t\in\left[  a,t_{0}\right]  $ (recall that $p_{n,k}$ is positive, so
it must increase at first). For $t\in\left[  t_{0},t_{1}\right]  $ we have
$p_{n,k}^{\prime}\left(  t\right)  <0$ since the zeros are simple. Finally we
have $p_{n,k}^{\prime}\left(  t\right)  >0$ for $t\in\left[  t_{1},b\right]
.$ So $p_{n,k}$ is increasing on $\left[  t_{1},b\right]  ,$ which implies
$p_{n,k}\left(  t_{1}\right)  \leq p_{n,k}\left(  b\right)  =0$ (recall that
$k\leq n-1).$ So $p_{n,k}$ has an additional zero in $t_{1}$ which is a
contradiction. Now assume that $p_{n,k}^{\prime}$ has a double zero at
$t_{0}.$ Since it has no further zeros $p_{n,k}^{\prime}\left(  t\right)  >0$
for all $t\in\left(  a,b\right)  , $ so $p_{n,k}$ is monotone increasing,
which gives a contradiction to the fact that $p_{n,k}\left(  b\right)  =0.$

In the next case we assume that $p_{n,k}^{\prime}$ has no zero in $\left(
a,b\right)  .$ Then $p_{n,k}$ is strictly increasing, so $p_{n,k}\left(
b\right)  >0,$ which gives a contradiction. So we see that $p_{n,k}^{\prime}$
has exactly one zero in $\left(  a,b\right)  .$

5. Since $p_{n,n}^{\prime}$ has a zero order $n-1$ at $a,$ it has at most one
zero in $\left(  a,b\right)  .$ If $p_{n,n}^{\prime}$ has a zero then
$p_{n,n}$ is increasing on $\left[  a,t_{0}\right]  $ and decreasing on
$\left[  t_{0},b\right]  .$ If $p_{n,n}^{\prime}$ has no zero then $p_{n,n}$
is increasing.
\end{proof}

We mention that $E_{\left(  \lambda_{0},...,\lambda_{n}\right)  }$ is an
extended Chebyshev system over \emph{intervals} $\left[  a,b\right]  $ whose
length $b-a$ is sufficiently small. Moreover, for \emph{real}
\emph{eigenvalues} $\lambda_{0},...,\lambda_{n}$ the space $E_{\left(
\lambda_{0},...,\lambda_{n}\right)  }$ is an extended Chebyshev system over
\emph{any} interval $\left[  a,b\right]  .$

Simple experiments show that the assumption of a Chebyshev system over the
interval $\left[  a,b\right]  $ are essential. In the case of the six
eigenvalues
\[
\pm7i,\pm i\left(  7-\pi\right)  ,\pm i.
\]
we have found basis function $p_{n,k}$ for the interval $\left[  0,3\right]  $
or $\left[  0,3.14\right]  $ with several relative maxima. Even it may be
possible that $p_{n,k}$ are non-negative (but they are always real-valued).
The fundamental function $\varphi_{6}\left(  x\right)  $ has its first zero at $3.2.$

\section{\label{S5}Recurrence relations of the Bernstein basis}

Let us recall that $\Lambda_{2s+1}:=\left(  s+1\text{ times }\lambda,s+1\text{
times }\mu\right)  $ and let us introduce the following notations
\begin{align*}
\Lambda_{2s+}  &  :=\left(  s+1\text{ times }\lambda,s\text{ times }%
\mu\right)  ,\\
\Lambda_{2s-}  &  :=\left(  s\text{ times }\lambda,s+1\text{ times }%
\mu\right)  .
\end{align*}
In this section we want to give recursive formula for the Bernstein basis
\[
p_{2s+1,k}:=p_{\Lambda_{2s+1},k}\text{ and }p_{2s+,k}:=p_{\Lambda_{2s+},k}.
\]

\begin{theorem}
\label{ThmR1}For each $k=0,...,2s-1$ the following recursion formula holds
\[
A_{\pm}p_{2s+1,k+2}\left(  x\right)  =x\cdot p_{2s-1,k}-\left(  k+1\right)
\cdot p_{2s\pm,k+1}\left(  x\right)
\]
where the constant $A_{\pm}$ is given through
\[
A_{\pm}=\left(  k+2\right)  p_{2s-1,k}^{\left(  k+1\right)  }\left(  0\right)
-\left(  k+1\right)  p_{2s\pm,k+1}^{\left(  k+2\right)  }\left(  0\right)  .
\]
\end{theorem}

\begin{proof}
Let us recall that $p_{\Lambda_{n},k}$ is the unique element in $E_{\Lambda
_{n}}$ which has $k$ zeros in $0$ and $n-k$ zeros in $1$ with the
normalization
\begin{equation}
k!\lim_{x\rightarrow0}\frac{p_{\Lambda_{n},k}\left(  x\right)  }{x^{k}%
}=p_{\Lambda_{n}}^{\left(  k\right)  }\left(  0\right)  =1. \label{eqnormaln}%
\end{equation}
Consider $f_{\pm}\left(  x\right)  =a\cdot x\cdot p_{2s-2,k}+b\cdot
p_{2s\pm,k+1}\left(  x\right)  $ for coefficients $a$ and $b$ which we want to
define later. Clearly $f_{\pm}$ is in $E_{\Lambda_{2s+1}}.$ Note that
$p_{2s-1,k}$ has a zero of order $2s-1-k$ in $1,$and that $p_{2s\pm
,k+1}\left(  x\right)  $ has a zero of order $2s-\left(  k+1\right)  $ in $1.$
Hence $f_{\pm}$ has a zero of order at least $2s-1-k$ in $1$ (just as
$p_{2s+1,k+2}).$ For similar reasons we see that $f_{\pm}$ has a zero of order
at least $k+1$ in $0.$ We choose now the constants $a$ and $b$ in such a way
that $f_{\pm}$ has a zero of order $k+2$ in $0$ and that $f$ satisfies the
normalization in (\ref{eqnormaln}). By the uniqueness we infer that $f_{\pm
}=p_{2s+1,k+2}.$ Clearly $f_{\pm}$ has a zero of order $k+2$ in $0$ if
\[
a\cdot\left(  \frac{d^{k+1}}{dx^{k+1}}\left[  x\cdot p_{2s-1,k}\right]
\right)  _{x=0}+b\cdot p_{2s\pm,k+1}^{\left(  k+1\right)  }\left(  0\right)
=0.
\]
Recall that $p_{2s\pm,k+1}^{\left(  k+1\right)  }\left(  0\right)  =1$ and
\[
p_{2s-1,k}\left(  x\right)  =\frac{1}{k!}x^{k}+\frac{p_{2s-1,k}^{\left(
k+1\right)  }\left(  0\right)  }{\left(  k+1\right)  !}x^{k+1}+\frac
{p_{2s-1,k}^{\left(  k+2\right)  }\left(  0\right)  }{\left(  k+2\right)
!}x^{k+1}+...
\]
It follows that
\[
\left(  \frac{d^{k+1}}{dx^{k+1}}\left[  x\cdot p_{2s-1,k}\left(  x\right)
\right]  \right)  _{x=0}=\frac{d^{k+1}}{dx^{k+1}}\left[  \frac{1}{k!}%
x^{k+1}+...\right]  _{x=0}=k+1.
\]
Thus $b=-a\left(  k+1\right)  .$ The normalization condition gives the second
equation
\[
1=p_{2s+1,k+2}^{\left(  k+2\right)  }\left(  0\right)  =a\cdot\left(
\frac{d^{k+2}}{dx^{k+2}}\left[  x\cdot p_{2s-1,k}\right]  \right)
_{x=0}-a\cdot\left(  k+1\right)  p_{2s\pm,k+1}^{\left(  k+2\right)  }\left(
0\right)  .
\]
Again we see that
\[
\frac{d^{k+2}}{dx^{k+2}}\left[  x\cdot p_{2s-1,k}\left(  x\right)  \right]
_{x=0}=\left(  k+2\right)  p_{2s-1,k}^{\left(  k+1\right)  }\left(  0\right)
.
\]
So we obtain
\[
1=a\left(  \left(  k+2\right)  p_{2s-1,k}^{\left(  k+1\right)  }\left(
0\right)  -\left(  k+1\right)  p_{2s\pm,k+1}^{\left(  k+2\right)  }\left(
0\right)  \right)  .
\]
The proof is complete.
\end{proof}

\begin{theorem}
\label{ThmR2}For each $k=0,...,2s-3$ holds the following recursion formula
\begin{equation}
A_{k}\cdot p_{2s+1,k+4}=x^{2}\cdot p_{2s-3,k}-\left(  k+1\right)  \left(
k+2\right)  p_{2s-1,k+2}+B_{k}\cdot x\cdot p_{2s-1,k+2} \label{eqBernrec}%
\end{equation}
where the constants $A_{k}$ and $B_{k}$ are defined by
\begin{align*}
B_{k}  &  =\frac{k+2}{k+3}\left(  \left(  k+1\right)  p_{2s-1,k+2}^{\left(
k+3\right)  }\left(  0\right)  -\left(  k+3\right)  p_{2s-3,k}^{\left(
k+1\right)  }\left(  0\right)  \right)  ,\\
A_{k}  &  =\left(  k+3\right)  \left(  k+4\right)  p_{2s-3,k}^{\left(
k+2\right)  }\left(  0\right)  -\left(  k+1\right)  \left(  k+2\right)
p_{2s-1,k+2}^{\left(  k+4\right)  }\left(  0\right) \\
&  +\left(  k+4\right)  p_{2s-1,k+2}^{\left(  k+3\right)  }\left(  0\right)
\cdot B_{k}.
\end{align*}
\end{theorem}

\begin{proof}
Consider $f\left(  x\right)  =a\cdot x^{2}\cdot p_{2s-3,k}+b\cdot
p_{2s-1,k+2}+c\cdot x\cdot p_{2s-1,k+2}$ for coefficients $a,b$ and $c$ which
we want to define later. Clearly $f$ is in $E_{\Lambda_{2s+1}}.$ Note that the
functions $p_{2s-3,k},p_{2s-1,k+2}$ and $p_{2s-1,k+2}$ have a zero of order at
least $2s-3-k$ in $1.$ Hence $f$ has a zero of order at least $2s-3-k$ in $1$
(just as $p_{2s+1,k+4}).$ Clearly $f$ has a zero of order at least $k+2$ in
$0.$ We choose now the constants $a,b$ and $c$ in such a way that $f$ has a
zero of order $k+4$ in $0$ and that $f$ satisfies the normalization in
(\ref{eqnormaln}). By the uniqueness we infer that $f=p_{2s+1,k+4}.$ Clearly
$f$ has a zero of order $k+4$ in $0$ if and only if
\[
f^{\left(  k+2\right)  }\left(  0\right)  =f^{\left(  k+3\right)  }\left(
0\right)  =0.
\]
Recall that $p_{2s-1,k+2}^{\left(  k+2\right)  }\left(  0\right)  =1.$ Since
$x\cdot p_{2s-1,k+2}$ has a zero of order $k+3$ in $0$ the equation
$f^{\left(  k+2\right)  }\left(  0\right)  =0$ is equivalent to
\[
a\cdot\left(  \frac{d^{k+2}}{dx^{k+2}}\left[  x^{2}\cdot p_{2s-3,k}\right]
\right)  _{x=0}+b=0.
\]
Consider the Taylor expansions of $p_{2s-3,k}\left(  x\right)  $ and
$p_{2s-1,k+2}$
\begin{align*}
p_{2s-3,k}\left(  x\right)   &  =\frac{1}{k!}x^{k}+\frac{p_{2s-3,k}^{\left(
k+1\right)  }\left(  0\right)  }{\left(  k+1\right)  !}x^{k+1}+\frac
{p_{2s-3,k}^{\left(  k+2\right)  }\left(  0\right)  }{\left(  k+2\right)
!}x^{k+2}+...\\
p_{2s-1,k+2}  &  =\frac{1}{\left(  k+2\right)  !}x^{k+2}+\frac{p_{2s-1,k+2}%
^{\left(  k+3\right)  }\left(  0\right)  }{\left(  k+3\right)  !}x^{k+3}+.....
\end{align*}
It follows that
\[
\left(  \frac{d^{k+2}}{dx^{k+2}}\left[  x^{2}\cdot p_{2s-3,k}\left(  x\right)
\right]  \right)  _{x=0}=\frac{d^{k+2}}{dx^{k+2}}\left[  \frac{1}{k!}%
x^{k+2}\right]  _{x=0}=\left(  k+1\right)  \left(  k+2\right)  .
\]
Hence $b=-a\left(  k+1\right)  \left(  k+2\right)  .$ The equation $f^{\left(
k+3\right)  }\left(  0\right)  =0$ implies that
\[
a\cdot\left(  \frac{d^{k+3}}{dx^{k+3}}\left[  x^{2}\cdot p_{2s-3,k}\right]
\right)  _{x=0}+bp_{2s-1,k+2}^{\left(  k+3\right)  }+c\left(  \frac{d^{k+3}%
}{dx^{k+3}}\left[  x\cdot p_{2s-1,k+2}\right]  \right)  _{x=0}=0.
\]
Clearly
\begin{align*}
\left(  \frac{d^{k+3}}{dx^{k+3}}\left[  x^{2}\cdot p_{2s-3,k}\right]  \right)
_{x=0}  &  =\left(  \frac{d^{k+3}}{dx^{k+3}}\frac{p_{2s-3,k}^{\left(
k+1\right)  }\left(  0\right)  }{\left(  k+1\right)  !}x^{k+3}\right)
_{x=0}\\
&  =\left(  k+3\right)  \left(  k+2\right)  p_{2s-3,k}^{\left(  k+1\right)
}\left(  0\right)
\end{align*}
and similarly
\[
\left(  \frac{d^{k+3}}{dx^{k+3}}\left[  x\cdot p_{2s-1,k+2}\right]  \right)
_{x=0}=\left(  \frac{d^{k+3}}{dx^{k+3}}\frac{1}{\left(  k+2\right)  !}%
x^{k+3}\right)  _{x=0}=\left(  k+3\right)  .
\]
So with $b=-a\left(  k+1\right)  \left(  k+2\right)  $ we obtain the equation
\[
a\left(  k+3\right)  \left(  k+2\right)  p_{2s-3,k}^{\left(  k+1\right)
}\left(  0\right)  -a\left(  k+1\right)  \left(  k+2\right)  p_{2s-1,k+2}%
^{\left(  k+3\right)  }\left(  0\right)  +c\left(  k+3\right)  =0,
\]
which shows that
\[
c\left(  k+3\right)  =a\left(  k+2\right)  \left(  \left(  k+1\right)
p_{2s-1,k+2}^{\left(  k+3\right)  }\left(  0\right)  -\left(  k+3\right)
p_{2s-3,k}^{\left(  k+1\right)  }\left(  0\right)  \right)  .
\]
so $c=B_{k}.$ The normalization condition gives the third equation
\begin{align*}
1  &  =p_{2s+1,k+4}^{\left(  k+4\right)  }\left(  0\right)  =a\cdot\left(
\frac{d^{k+4}}{dx^{k+4}}\left[  x^{2}\cdot p_{2s-3,k}\left(  x\right)
\right]  \right)  _{x=0}\\
&  -a\left(  k+1\right)  \left(  k+2\right)  p_{2s-1,k+2}^{\left(  k+4\right)
}\left(  0\right)  +c\left(  \frac{d^{k+4}}{dx^{k+4}}\left[  x\cdot
p_{2s-1,k+2}\left(  x\right)  \right]  \right)  _{x=0}.
\end{align*}
Again we see that
\begin{align*}
\left(  \frac{d^{k+4}}{dx^{k+4}}\left[  x^{2}\cdot p_{2s-3,k}\left(  x\right)
\right]  \right)  _{x=0}  &  =\left(  \frac{d^{k+4}}{dx^{k+4}}\left[
\frac{p_{2s-3,k}^{\left(  k+2\right)  }\left(  0\right)  }{\left(  k+2\right)
!}x^{k+4}\right]  \right)  _{x=0}\\
&  =\left(  k+3\right)  \left(  k+4\right)  p_{2s-3,k}^{\left(  k+2\right)
}\left(  0\right)  ,
\end{align*}
and
\begin{align*}
\left(  \frac{d^{k+4}}{dx^{k+4}}\left[  x\cdot p_{2s-1,k+2}\left(  x\right)
\right]  \right)  _{x=0}  &  =\left(  \frac{d^{k+4}}{dx^{k+4}}\left[
\frac{p_{2s-1,k+2}^{\left(  k+3\right)  }\left(  0\right)  }{\left(
k+3\right)  !}x^{k+4}\right]  \right)  _{x=0}\\
&  =\left(  k+4\right)  p_{2s-1,k+2}^{\left(  k+3\right)  }\left(  0\right)  .
\end{align*}
So we arrive at the equation
\begin{align*}
a^{-1}  &  =\left(  k+3\right)  \left(  k+4\right)  p_{2s-3,k}^{\left(
k+2\right)  }\left(  0\right)  -\left(  k+1\right)  \left(  k+2\right)
p_{2s-1,k+2}^{\left(  k+4\right)  }\left(  0\right) \\
&  +\left(  k+4\right)  p_{2s-1,k+2}^{\left(  k+3\right)  }\left(  0\right)
B_{k}.
\end{align*}
\end{proof}

\begin{corollary}
\label{CorRec2}Suppose that $\mu=-\lambda.$ Then the fundamental function
$\Phi_{2s+1}\left(  x\right)  $ satisfies the following recursion
\[
\Phi_{2s+1}\left(  x\right)  =\frac{1}{4s\left(  s-1\right)  }x^{2}\cdot
\Phi_{2s-3}\left(  x\right)  -\frac{2s-1}{2s}\Phi_{2s-1}\left(  x\right)  .
\]
\end{corollary}

\begin{proof}
Let us take $k=2s-3$ in equation (\ref{eqBernrec}): then $p_{2s+1,2s+1}%
=\Phi_{2s+1}$ and $p_{2s-3,2s-3}=\Phi_{2s-3}$ and $p_{2s-1,2s-1}=\Phi_{2s-1}$.
By Proposition \ref{PropTaylor}
\[
\Phi_{2s-3}^{\left(  2s-2\right)  }\left(  0\right)  =\left(  s-2\right)
\lambda+\left(  s-2\right)  \mu=0.
\]
and similarly $\Phi_{2s-1}^{\left(  2s\right)  }=0.$ Hence for $k=2s-3$ the
last summand in (\ref{eqBernrec}) is zero and the formula amounts to
\[
A_{2s-3}\cdot\Phi_{2s+1}=x^{2}\cdot\Phi_{2s-3}-\left(  2s-2\right)  \left(
2s-1\right)  \Phi_{2s-1}.
\]
We have to compute the constant $A_{k}$ for $k=2s-3.$ Since $\Phi
_{2s-1}^{\left(  2s\right)  }\left(  0\right)  =0$ and $\Phi_{2s-3}^{\left(
2s-2\right)  }\left(  0\right)  =0$ the constant $B_{k}$ is zero and we have
\[
A_{2s-3}=2s\left(  2s+1\right)  \Phi_{2s-3}^{\left(  2s-1\right)  }\left(
0\right)  -\left(  2s-2\right)  \left(  2s-1\right)  \Phi_{2s-1}^{\left(
2s+1\right)  }\left(  0\right)  .
\]
Proposition \ref{PropTaylor3} shows that $\Phi_{2s+1}^{\left(  2s+3\right)
}\left(  0\right)  =s+1$, so $\Phi_{2s-3}^{\left(  2s-1\right)  }\left(
0\right)  =s-1$ and we see that
\begin{align*}
A_{2s-3}  &  =2s\left(  2s+1\right)  \left(  s-1\right)  -\left(  2s-2\right)
\left(  2s-1\right)  s\\
&  =2s\left(  s-1\right)  \left(  2s+1-2s+1\right)  =4s\left(  s-1\right)  .
\end{align*}
Hence
\[
\Phi_{2s+1}=\frac{1}{4s\left(  s-1\right)  }x^{2}\cdot\Phi_{2s-3}-\frac
{2s-1}{2s}\Phi_{2s-1}.
\]
\end{proof}

Let us use the notation $q_{2s+1,l}=p_{2s+1,2s+1-l}.$ Then the recurrence
relation in this notation means (with $l=2s+1-\left(  k+4\right)  $, so
$k=2s-3-l)$ that
\[
A_{2s-3-l}\cdot q_{2s+1,l}=x^{2}\cdot q_{2s-3,l}-\left(  2s-2-l\right)
\left(  2s-1-l\right)  q_{2s-1,l}+B_{2s-3-l}\cdot x\cdot q_{2s-1,l},
\]
so we have a recurrence relation for fixed $l.$

\section{\label{S6}The derivative of the Bernstein operator}

Let $\mathbb{C}^{\mathbb{N}_{0}}$ be the set of all sequences $y=\left(
y_{0},y_{1},....\right)  =\left(  y_{k}\right)  $ with complex entries. For
$y\in\mathbb{C}^{\mathbb{N}_{0}}$ we shall use the notation
\[
\left(  y\right)  _{k}:=y_{k}%
\]
for describing the $k$-th component of the vector $y.$ The finite difference
operator $\Delta:\mathbb{C}^{\mathbb{N}_{0}}\rightarrow\mathbb{C}%
^{\mathbb{N}_{0}}$ is defined for $y=\left(  y_{0},y_{1},....\right)  =\left(
y_{k}\right)  \in\mathbb{C}^{\mathbb{N}_{0}}$ by
\[
\left(  \Delta y\right)  _{k}:=y_{k+1}-y_{k}.
\]
Higher differences are defined inductively by setting $\Delta^{n+1}%
y=\Delta\left(  \Delta^{n}y\right)  $ where $\Delta^{0}$ is defined as the
identity operator. The difference operator $\Delta$ is useful in the classical
theory to describe the derivative of the Bernstein polynomial
\[
B_{n}f\left(  x\right)  =\sum_{k=0}^{n}f\left(  \frac{k}{n}\right)  \binom
{n}{k}x^{k}\left(  1-x\right)  ^{n-k},
\]
namely
\[
\frac{d}{dx}(B_{n}f)(x)=n\sum_{k=0}^{n-1}\Delta\left[  f\left(  \frac{k}%
{n}\right)  \right]  \binom{n-1}{k}x^{k}\left(  1-x\right)  ^{n-1-k}.
\]
We want to derive an analog for the Bernstein exponential polynomial.

We recall from \cite{AKR} the following numbers
\[
d_{\Lambda_{n}}^{k,\lambda_{j}}:=\lim_{x\rightarrow b}\frac{\frac{d}%
{dx}p_{\Lambda_{n},k}\left(  x\right)  }{p_{\Lambda_{n}\setminus\lambda_{j}%
,k}\left(  x\right)  }%
\]
which have been important for the construction of the Bernstein operators. We
shall assume that $\lambda_{0}=0$ which facilitates the formulas. We define a
difference operator $\Delta_{\Lambda_{n},\lambda_{j}}:\mathbb{C}%
^{\mathbb{N}_{0}}\rightarrow\mathbb{C}^{\mathbb{N}_{0}}$ for $y=\left(
y_{k}\right)  \in\mathbb{C}^{\mathbb{N}_{0}}$ by
\[
\left(  \Delta_{\Lambda_{n},\lambda_{j}}y\right)  _{k}:=d_{\Lambda_{n}%
}^{k,\lambda_{j}}\cdot y_{k}-d_{\Lambda_{n}}^{k,\lambda_{0}}\cdot y_{k+1}%
\]
By $\Lambda_{n}\setminus\lambda_{j}$ we denote the vector where we have
deleted $\lambda_{j}$. Now we can prove

\begin{theorem}
Let $\lambda_{0}=0$ and $j\in\left\{  0,1,...,n\right\}  .$ For $B_{\left(
\lambda_{0},...,\lambda_{n}\right)  }f$ the following identity holds:
\[
(\frac{d}{dx}-\lambda_{j})B_{\left(  \lambda_{0},...,\lambda_{n}\right)
}f\left(  x\right)  =\sum_{k=0}^{n-1}\Delta_{\Lambda_{n},\lambda_{j}}\left[
f\left(  t_{k}\right)  \right]  \cdot\alpha_{k}\cdot p_{\Lambda_{n}%
\setminus\lambda_{j},k}\left(  x\right)
\]
\end{theorem}

\begin{proof}
From the definition of the Bernstein operator we immediately see that
\[
(\frac{d}{dx}-\lambda_{j})B_{\left(  \lambda_{0},...,\lambda_{n}\right)
}f\left(  x\right)  =\sum_{k=0}^{n}\alpha_{k}f\left(  t_{k}\right)  (\frac
{d}{dx}-\lambda_{j})p_{\left(  \lambda_{0},...,\lambda_{n}\right)  ,k}.
\]
By the next Theorem below we obtain
\begin{align*}
(\frac{d}{dx}-\lambda_{j})B_{\left(  \lambda_{0},...,\lambda_{n}\right)
}f\left(  x\right)   &  =\sum_{k=1}^{n}\alpha_{k}f\left(  t_{k}\right)
p_{\Lambda_{n}\setminus\lambda_{j},k-1}+\sum_{k=0}^{n-1}\alpha_{k}f\left(
t_{k}\right)  d_{\Lambda_{n}}^{k,\lambda_{j}}p_{\Lambda_{n}\setminus
\lambda_{j},k}\\
&  =\sum_{k=0}^{n-1}\left[  \alpha_{k}f\left(  t_{k}\right)  d_{\Lambda_{n}%
}^{k,\lambda_{j}}+\alpha_{k+1}f\left(  t_{k+1}\right)  \right]  \cdot
p_{\Lambda_{n}\setminus\lambda_{j},k}%
\end{align*}
From the construction of the Bernstein operator in \cite{AKR} the following
formula is known:
\[
\alpha_{k+1}=-\alpha_{k}d_{\Lambda_{n}}^{k,\lambda_{0}}.
\]
The proof is accomplished by identity
\[
\alpha_{k+1}f\left(  t_{k+1}\right)  +\alpha_{k}f\left(  t_{k}\right)
d_{\Lambda_{n}}^{k,\lambda_{j}}=\alpha_{k}\left[  -d_{\Lambda_{n}}%
^{k,\lambda_{0}}f\left(  t_{k+1}\right)  +f\left(  t_{k}\right)
d_{\Lambda_{n}}^{k,\lambda_{j}}\right]  .
\]
\end{proof}

From \cite{AKR} we repeat the following result:

\begin{proposition}
\label{PropABL} Define for $k=0,....,n-1$
\begin{equation}
d_{\left(  \lambda_{0},...,\lambda_{n}\right)  }^{k,\lambda_{n}}%
:=\lim_{x\uparrow b}\frac{\frac{d}{dx}p_{\left(  \lambda_{0},...,\lambda
_{n}\right)  ,k}\left(  x\right)  }{p_{\left(  \lambda_{0},...,\lambda
_{n-1}\right)  ,k}\left(  x\right)  } \label{eqPR}%
\end{equation}
Then, for any $k=1,...,n-1,$
\begin{equation}
\left(  \frac{d}{dx}-\lambda_{n}\right)  p_{\left(  \lambda_{0},...,\lambda
_{n}\right)  ,k}=p_{\left(  \lambda_{0},...,\lambda_{n-1}\right)
,k-1}+d_{\left(  \lambda_{0},...,\lambda_{n}\right)  }^{k,\lambda_{n}%
}p_{\left(  \lambda_{0},...,\lambda_{n-1}\right)  ,k}. \label{eqPREC}%
\end{equation}
Furthermore, for $k=0$ we have
\begin{equation}
\left(  \frac{d}{dx}-\lambda_{n}\right)  p_{\left(  \lambda_{0},...,\lambda
_{n}\right)  ,0}=d_{\left(  \lambda_{0},...,\lambda_{n}\right)  }^{\lambda
_{n},0}p_{\left(  \lambda_{0},...,\lambda_{n-1}\right)  ,0}, \label{neu1}%
\end{equation}
while for $k=n$,
\begin{equation}
\left(  \frac{d}{dx}-\lambda_{n}\right)  p_{\left(  \lambda_{0},...,\lambda
_{n}\right)  ,n}=p_{\left(  \lambda_{0},...,\lambda_{n-1}\right)  ,n-1}.
\label{neu2}%
\end{equation}
\end{proposition}

\begin{proof}
We may assume that $a=0.$ Set
\[
f_{k}:=\left(  \frac{d}{dx}-\lambda_{n}\right)  p_{\left(  \lambda
_{0},...,\lambda_{n}\right)  ,k},
\]
and let $1\leq k\leq n-1.$ Using the fact that $f_{k}$ has a zero of order
$k-1$ at $0$ and of order $n-k-1$ at $b$, it is easy to see that $f_{k}%
=c_{k}p_{\left(  \lambda_{0},...,\lambda_{n-1}\right)  ,k-1}+d_{k}p_{\left(
\lambda_{0},...,\lambda_{n-1}\right)  ,k}$ for some constants $c_{k}$ and
$d_{k}.$ We want to show that $c_{k}=1.$ Note that $p_{\left(  \lambda
_{0},...,\lambda_{n}\right)  ,k}$ has a zero of order $k$ in $0,$ so
\[
\lim_{x\downarrow0}\frac{f_{k}\left(  x\right)  }{x^{k-1}}=\lim_{x\downarrow
0}\frac{1}{x^{k-1}}\frac{d}{dx}p_{\left(  \lambda_{0},...,\lambda_{n}\right)
,k}\left(  x\right)  =\frac{p_{\left(  \lambda_{0},...,\lambda_{n}\right)
,k}^{\left(  k\right)  }\left(  0\right)  }{\left(  k-1\right)  !}=\frac
{1}{\left(  k-1\right)  !}%
\]
where the second equality follows from (\ref{eqLim}) applied to $p_{\left(
\lambda_{0},...,\lambda_{n}\right)  ,k}^{\left(  1\right)  }\left(  x\right)
$ and (\ref{eqnorm}). On the other hand, the equation $f_{k}=c_{k}p_{\left(
\lambda_{0},...,\lambda_{n-1}\right)  ,k-1}+d_{k}p_{\left(  \lambda
_{0},...,\lambda_{n-1}\right)  ,k}$ shows that
\[
\lim_{x\downarrow0}\frac{f_{k}\left(  x\right)  }{x^{k-1}}=c_{k}%
\lim_{x\downarrow0}\frac{p_{\left(  \lambda_{0},...,\lambda_{n-1}\right)
,k-1}\left(  x\right)  }{x^{k-1}}=c_{k}\frac{1}{\left(  k-1\right)  !},
\]
using again Proposition \ref{PropExist}, (\ref{eqnorm}). Hence $c_{k}=1.$ Next
we divide $f_{k}$ by $\left(  b-x\right)  ^{n-k-1}$ to obtain
\begin{align*}
&  \lim_{x\uparrow b}\frac{d_{k}p_{\left(  \lambda_{0},...,\lambda
_{n-1}\right)  ,k}\left(  x\right)  }{\left(  b-x\right)  ^{n-k-1}}%
=\lim_{x\uparrow b}\frac{f_{k}(x)}{\left(  b-x\right)  ^{n-k-1}}\\
&  =\lim_{x\uparrow b}\frac{\frac{d}{dx}p_{\left(  \lambda_{0},...,\lambda
_{n}\right)  ,k}\left(  x\right)  -\lambda_{n}p_{\left(  \lambda
_{0},...,\lambda_{n}\right)  ,k}(x)}{\left(  b-x\right)  ^{n-k-1}}\\
&  =\lim_{x\uparrow b}\frac{\frac{d}{dx}p_{\left(  \lambda_{0},...,\lambda
_{n}\right)  ,k}\left(  x\right)  }{\left(  b-x\right)  ^{n-k-1}}.
\end{align*}
The case $k=0$ follows by noticing that $f_{0}=d_{0}p_{\left(  \lambda
_{0},...,\lambda_{n-1}\right)  ,0}$, solving for $d_{0}$ and taking the limit
as $x\uparrow b$, while the case $k=n$ is an immediate consequence of the fact
that $p_{\left(  \lambda_{0},...,\lambda_{n}\right)  ,n}=\Phi_{\Lambda_{n}}$.
\end{proof}

\end{document}